\documentclass[draft,11pt]{amsart}
\usepackage{amsmath,amssymb,amsthm,bm}
\usepackage{layout}

\theoremstyle{plain}
\newtheorem{theorem}{Theorem}[section]
\newtheorem{lemma}[theorem]{Lemma}
\newtheorem{proposition}[theorem]{Proposition}
\newtheorem{corollary}[theorem]{Corollary}

\theoremstyle{definition}
\newtheorem*{definition}{Definition}
\newtheorem{example}[theorem]{Example}
\newtheorem{remark}[theorem]{Remark}
\newtheorem*{subfield-problem}{Subfield problem of a generic polynomial}
\newtheorem*{intersection-problem}{Field intersection problem of a generic polynomial}
\newtheorem*{isomorphism-problem}{Field isomorphism problem of a generic polynomial}
\newtheorem*{acknowledgment}{Acknowledgment}

\newcommand{\bs}{\mathbf{s}}\newcommand{\bt}{\mathbf{t}}
\newcommand{\ba}{\mathbf{a}}\newcommand{\bb}{\mathbf{b}}
\newcommand{\bp}{\mathbf{p}}
\newcommand{\bx}{\mathbf{x}}\newcommand{\by}{\mathbf{y}}

\newcommand{\Gs}{G_\mathbf{s}}\newcommand{\Gt}{G_\mathbf{t}}
\newcommand{\Gst}{G_{\mathbf{s},\mathbf{t}}}
\newcommand{\Hst}{H_{\mathbf{s},\mathbf{t}}}

\newcommand{\opi}{\overline{\pi}}

\title[On the field intersection problem of solvable quintic generic polynomials]
{On the field intersection problem of solvable quintic generic polynomials}
\author{Akinari Hoshi and Katsuya Miyake}

\thanks{This work was partially supported by Grant-in-Aid for Scientific 
Research (C) 19540057 of Japan Society for the Promotion of Science}



\begin{document}
\maketitle
\begin{abstract}
We study a general method of the field intersection problem of generic polynomials over 
an arbitrary field $k$ via formal Tschirnhausen transformation. 
In the case of solvable quintic, we give an explicit answer to the problem by using 
multi-resolvent polynomials. 
\end{abstract}
\section{Introduction}\label{seIntro}
Let $G$ be a finite group. 
Let $k$ be an arbitrary field and $k(\bt)$ the rational function field over $k$ with 
$n$ indeterminates $\bt=(t_1,\ldots,t_n)$. 

A polynomial $f_\bt(X)\in k(\bt)[X]$ is called $k$-generic for $G$ if it has the following 
property: the Galois group of $f_\bt(X)$ over $k(\bt)$ is isomorphic to $G$ and every 
$G$-Galois extension $L/M, \# M=\infty, M\supset k,$ can be obtained as 
$L=\mathrm{Spl}_M f_\ba(X)$, the splitting field of $f_\ba(X)$ over $M$, for some 
$\ba=(a_1,\ldots,a_n)\in M^n$. 
We suppose that the base field $M$, $M\supset k$, of a $G$-extension $L/M$ is an infinite field. 

Examples of $k$-generic polynomials for $G$ are known for various pairs of $(k,G)$ 
(for example, see \cite{Kem94}, \cite{KM00}, \cite{JLY02}, \cite{Rik04}). 
Let $f_\bt^G(X)\in k(\bt)[X]$ be a $k$-generic polynomial for $G$. 
Since a $k$-generic polynomial $f_\bt^G(X)$ for $G$ covers all $G$-Galois extensions over 
$M\supset k$ by specializing parameters, it is natural to ask the following problem: 

\begin{isomorphism-problem}
Determine whether $\mathrm{Spl}_M f_\ba^G(X)$ and $\mathrm{Spl}_M f_\bb^G(X)$ are 
isomorphic over $M$ or not for $\ba,\bb\in M^n$. 
\end{isomorphism-problem}

It would be desired to give an answer to the problem within the base field $M$ by using 
the data $\ba,\bb\in M^n$. 
Let $S_n$ (resp. $D_n, C_n$) be the symmetric (resp. the dihedral, the cyclic) group of 
degree $n$. 
For $C_3$, the polynomial $f_t^{C_3}(X)=X^3-tX^2-(t+3)X-1\in k(t)[X]$ is $k$-generic 
for an arbitrary field $k$. 
We showed in \cite{HM} the following theorem which is an analogue to the results of 
Morton \cite{Mor94} and Chapman \cite{Cha96}. 
\begin{theorem}[\cite{Mor94}, \cite{Cha96}, \cite{HM}]\label{thC3}
Let $f_t^{C_3}(X)$ be as above and assume char $k\neq 2$. 
For $m,n\in M$ with $(m^2+3m+9)(n^2+3n+9)\neq 0$, 
two splitting fields $\mathrm{Spl}_M f_m^{C_3}(X)$ and 
$\mathrm{Spl}_M f_n^{C_3}(X)$ over $M$ coincide if and only if 
there exists $z\in M$ such that either 
\begin{align*}
n\,=\,\frac{m(z^3-3z-1)-9z(z+1)}{mz(z+1)+z^3+3z^2-1}\ \mathit{or}\ 
n\,=\,-\frac{m(z^3+3z^2-1)+3(z^3-3z-1)}{mz(z+1)+z^3+3z^2-1}.
\end{align*}
\end{theorem}
We have $\mathrm{Spl}_{M} f_m^{C_3}(X)=\mathrm{Spl}_{M} f_n^{C_3}(X)$ 
whenever $m, n\in M$ satisfy the condition in Theorem \ref{thC3}. 
In particular, over an infinite field $M$, for each fixed $m\in M$ 
with $\mathrm{Gal}(f_m^{C_3}/M)=\{1\}$ or $\mathrm{Gal}(f_m^{C_3}/M)=C_3$ 
there exist infinitely many $n\in M$ such that 
$\mathrm{Spl}_M f_m^{C_3}(X)=\mathrm{Spl}_M f_n^{C_3}(X)$. 

We also gave analogues to Theorem \ref{thC3} for two non-abelian groups 
for $S_3$ and $D_4$ in \cite{HM07} and \cite{HM-2} respectively as follows: 
\begin{theorem}[\cite{HM07}]
Let $k$ be a field of char $k\neq 3$ and $f_t^{S_3}(X)=X^3+tX+t\in k(t)[X]$ a $k$-generic 
polynomial for $S_3$. 
For $a,b\in M\setminus\{0,-27/4\}$ 
with $a\neq b$, two splitting fields $\mathrm{Spl}_M f_a^{S_3}(X)$ and 
$\mathrm{Spl}_M f_b^{S_3}(X)$ over $M$ coincide if and only if 
there exists $u\in M$ such that 
\[
b=\frac{a(u^2+9u-3a)^3}{(u^3-2au^2-9au-2a^2-27a)^2}.
\]
\end{theorem}

\begin{theorem}[\cite{HM-2}]\label{lemD4spl}
Let $k$ be a field of char $k\neq 2$ and $f_{s,t}^{D_4}(X)=X^4+sX^2+t\in k(s,t)[X]$ 
a $k$-generic polynomial for $D_4$. 
For $a,b\in M$, we assume that $\mathrm{Gal}(f_{a,b}^{D_4}/M)=D_4$. 
Then for $a,b,a',b'\in M$, two splitting fields $\mathrm{Spl}_M f_{a,b}^{D_4}(X)$ and 
$\mathrm{Spl}_M f_{a',b'}^{D_4}(X)$ over $M$ coincide if and only if 
there exist $p,q\in M$ such that either 
\begin{align*}
{\rm (i)}\quad a'&=ap^2-4bpq+abq^2,\quad b'=b(p^2-apq+bq^2)^2\quad \mathit{or}\\
{\rm (ii)}\quad  a'&=2(ap^2-4bpq+abq^2),\quad b'=(a^2-4b)(p^2-bq^2)^2.
\end{align*}
\end{theorem}

In Theorem \ref{lemD4spl}, under the assumption 
$C_4\leq \mathrm{Gal}(f_{a,b}^{D_4}/M)\leq D_4$, there exist $p,q\in M$ which satisfy 
the condition (i) if and only if 
$M[X]/(f_{a,b}^{D_4}(X))\cong_M M[X]/(f_{a',b'}^{D_4}(X))$ (cf. \cite[Lemma 4.14]{HM-2}), 
and this fact was given by Van der Ploeg \cite{Plo87} when $M=\mathbb{Q}$. 

As in the case of $C_3$ over an infinite field $M$, for a fixed $a\in M$ with 
$\mathrm{Gal}(f_a^{S_3}/M)\leq S_3$ (resp. fixed $a,b \in K$ with 
$C_4\leq \mathrm{Gal}(f_{a,b}^{D_4}/M)\leq D_4$) there exist infinitely many $b\in M$ 
(resp. $a',b'\in M$) such that $\mathrm{Spl}_M f_a^{S_3}(X)=\mathrm{Spl}_M f_b^{S_3}(X)$ 
(resp. $\mathrm{Spl}_M f_{a,b}^{D_4}(X)=\mathrm{Spl}_M f_{a',b'}^{D_4}(X)$). \\

Kemper \cite{Kem01}, furthermore, showed that for a subgroup $H$ of $G$ every $H$-Galois 
extension over $M\supset k$ is also given by a specialization of $f_\bt^G(X)$ as in a 
similar manner. 
Hence the following two problems naturally arise: 

\begin{intersection-problem}
For a field $M\supset k$ and $\ba,\bb\in M^n$, determine 
the intersection of $\mathrm{Spl}_M f_\ba^G(X)$ and $\mathrm{Spl}_M f_\bb^G(X)$. 
\end{intersection-problem}

\begin{subfield-problem}
For a field $M\supset k$ and $\ba,\bb\in M^n$, determine whether 
$\mathrm{Spl}_M f_\bb^G(X)$ is a subfield of $\mathrm{Spl}_M f_\ba^G(X)$ or not. 
\end{subfield-problem}

If we get an answer to the field intersection problem of a $k$-generic polynomial, 
we obtain an answer to the subfield problem and hence that of the field isomorphism problem. 

The aim of this paper is to study a method to give an answer to the intersection problem of 
$k$-generic polynomials via formal Tschirnhausen transformation and multi-resolvent polynomials. 

In Section \ref{seResolv}, we review some known results about resolvent polynomials. 
In Section \ref{seTschin}, we recall formal Tschirnhausen transformation which is given 
in \cite{HM}. 
In Section \ref{seInt}, by using materials given in Sections \ref{seResolv} and \ref{seTschin}, 
we give a general method to solve the intersection problem of $k$-generic polynomials. 

In Section \ref{seGenPoly}, we give a general method to construct a generic polynomial for the 
direct product $H_1 \times H_2$ of two subgroups $H_1$ and $H_2$ of $S_n$. 

In Section \ref{seQuin}, we take the following quintic generic polynomials 
with two parameters for $F_{20}, D_5$ and $C_5$, respectively, where $F_{20}$ is the Frobenius 
group of order $20$. 
\begin{align*}
f_{p,q}^{F_{20}}(X)&=X^5+\Bigl(\frac{q^2+5pq-25}{p^2+4}-2p+2\Bigr)X^4\\
&\hspace*{11mm} +\bigl(p^2-p-3q+5\bigr)X^3+(q-3p+8)X^2+(p-6)X+1\ \in\, k(p,q)[X],\\
f_{s,t}^{D_5}(X)&=X^5+(t-3)X^4+(s-t+3)X^3+(t^2-t-2s-1)X^2+sX+t\ \in\, k(s,t)[X],\\
h_{A,B}^{C_5}(X)&=X^5-\frac{P}{Q^2}(A^2-2A+15B^2+2)X^3+\frac{P^2}{Q^3}(2BX^2-(A-1)X-2B)
\ \in\, k(A,B)[X]
\end{align*}
where $P=(A^2-A-1)^2+25(A^2+1)B^2+125B^4$, $Q=1-A+7B^2+AB^2$. 

Based on the general method, we illustrate the solvable quintic cases and give an explicit 
answer to the problem by multi-resolvent polynomials in the final Section \ref{seIntQuin}. 
We also give some numerical examples. 

\section{Resolvent polynomial}\label{seResolv}

In this section we review known results in the computational aspects of Galois theory 
(cf. the text books \cite{Coh93}, \cite{Ade01}). 
One of the fundamental tools to determine the Galois group of a polynomial is the resolvent 
polynomials; an absolute resolvent polynomial was first introduced by Lagrange \cite{Lag1770}, 
and a relative resolvent polynomial by Stauduhar \cite{Sta73}. 
Several kinds of methods to compute resolvent polynomials have been developed by 
many mathematicians 
(see, for example, \cite{Sta73}, \cite{Gir83}, \cite{SM85}, \cite{Yok97}, \cite{MM97}, 
\cite{AV00}, \cite{GK00} and the references therein). 

Let $M\supset k$ be an infinite field and $\overline{M}$ a fixed algebraic closure of $M$.
Let $f(X):=\prod_{i=1}^m(X-\alpha_i) \in M[X]$ be a separable polynomial of degree $m$ with 
some fixed order of the roots $\alpha_1,\ldots,\alpha_m\in \overline{M}$. 
The Galois group of the splitting field $\mathrm{Spl}_M f(X)$ of $f(X)$ over $M$ 
may be obtained by using suitable resolvent polynomials. 

Let $k[\bx]:=k[x_1,\ldots,x_m]$ be the polynomial ring over $k$ with indeterminates 
$x_1,\ldots,x_m$. 
Put $R:=k[\bx, 1/\Delta_\bx]$ where $\Delta_\bx:=\prod_{1\leq i<j\leq m}(x_j-x_i)$. 
We take a surjective evaluation homomorphism $\omega_f : R \longrightarrow 
k(\alpha_1,\ldots,\alpha_m),\, \Theta(x_1,\ldots,x_m)\longmapsto 
\Theta(\alpha_1,\ldots,\alpha_m),$ for $\Theta \in R$.
We note that $\omega_f(\Delta_\bx)\neq 0$ from the assumption that $f(X)$ is separable. 
The kernel of the map $\omega_f$ is the ideal 
\[
I_f=\mathrm{ker}(\omega_f)=\{\Theta(x_1,\ldots,x_m)\in R 
\mid \Theta(\alpha_1,\ldots,\alpha_m)=0\}. 
\]
Let $S_m$ be the symmetric group of degree $m$. 
For $\pi\in S_m$, we extend the action of $\pi$ on $m$ letters $\{1,\ldots,m\}$ to 
that on $R$ by $\pi(\Theta(x_1,\ldots,x_m)):=\Theta(x_{\pi(1)},\ldots,x_{\pi(m)})$. 
We define the Galois group of a polynomial $f(X)\in M[X]$ over $M$ by 
\[
\mathrm{Gal}(f/M):=\{\pi\in S_m \mid \pi(I_f)\subseteq I_f\}.
\]
We write $\mathrm{Gal}(f):=\mathrm{Gal}(f/M)$ for simplicity. 
The Galois group of the splitting field $\mathrm{Spl}_M f(X)$ of a polynomial $f(X)$ over $M$ 
is isomorphic to $\mathrm{Gal}(f)$. 
If we take another ordering of roots $\alpha_{\pi(1)},\ldots,\alpha_{\pi(m)}$ of $f(X)$ 
with some $\pi\in S_m$, the corresponding realization of $\mathrm{Gal}(f)$ is the conjugate 
of the original one given by $\pi$ in $S_m$. 
Hence, for arbitrary ordering of the roots of $f(X)$, $\mathrm{Gal}(f)$ 
is determined up to conjugacy in $S_m$. 

\begin{definition}
For $H\leq G\leq S_m$, an element $\Theta\in R$ is called a $G$-primitive $H$-invariant if 
$H=\mathrm{Stab}_G(\Theta)$ $:=$ $\{\pi\in G\ |\ \pi(\Theta)=\Theta\}$. 
For a $G$-primitive $H$-invariant $\Theta$, the polynomial 
\[
\mathcal{RP}_{\Theta,G}(X):=\prod_{\opi\in G/H}(X-\pi(\Theta))\in R^G[X]
\]
where $\opi$ runs through the left cosets on $H$ in $G$, is called the {\it formal} 
$G$-relative $H$-invariant resolvent by $\Theta$, and the polynomial
\[
\mathcal{RP}_{\Theta,G,f}(X):=\prod_{\opi\in G/H}\bigl(X-\omega_f(\pi(\Theta))\bigr)
\]
is called the $G$-relative $H$-invariant resolvent of $f$ by $\Theta$. 
\end{definition}

The following theorem is fundamental in the theory of resolvent polynomials 
(cf. \cite[p.95]{Ade01}). 

\begin{theorem}\label{thfun}
For $H\leq G\leq S_m$, let $\Theta$ be a $G$-primitive  $H$-invariant. 
Assume that $\mathrm{Gal}(f)\leq G$. 
Suppose that $\mathcal{RP}_{\Theta,G,f}(X)$ is decomposed into a product of powers of distinct 
irreducible polynomials as $\mathcal{RP}_{\Theta,G,f}(X)=\prod_{i=1}^l h_i^{e_i}(X)$ in $M[X]$. 
Then we have a bijection 
\begin{align*}
\mathrm{Gal}(f)\backslash G/H\quad &\longrightarrow \quad \{h_1^{e_1}(X),\ldots,h_l^{e_l}(X)\}\\
\mathrm{Gal}(f)\, \pi\, H\quad &\longmapsto\quad h_\pi(X)
=\prod_{\tau H\subseteq \mathrm{Gal}(f)\,\pi\,H}\bigl(X-\omega_{f}(\tau(\Theta))\bigr)
\end{align*}
where the product runs through the left cosets $\tau H$ of $H$ in $G$ contained in 
$\mathrm{Gal}(f)\, \pi\, H$, that is, through $\tau=\pi_\sigma \pi$ where $\pi_\sigma$ runs 
a system of representatives of the left cosets of $\mathrm{Gal}(f) \cap \pi H\pi^{-1}$; 
each $h_\pi(X)$ is irreducible or a power of an irreducible polynomial with 
$\mathrm{deg}(h_\pi(X))$ $=$ $|\mathrm{Gal}(f)\, \pi\, H|/|H|$ $=$ 
$|\mathrm{Gal}(f)|/|\mathrm{Gal}(f)\cap \pi H\pi^{-1}|$. 
\end{theorem}

\begin{corollary} 
If $\mathrm{Gal}(f)\leq \pi H\pi^{-1}$ for some $\pi\in G$ then 
$\mathcal{RP}_{\Theta,G,f}(X)$ has a linear factor over $M$. 
Conversely, if $\mathcal{RP}_{\Theta,G,f}(X)$ has a non-repeated linear factor over $M$ 
then there exists $\pi\in G$ such that $\mathrm{Gal}(f)\leq \pi H\pi^{-1}$. 
\end{corollary}

\begin{remark}\label{remGir}
When the resolvent polynomial $\mathcal{RP}_{\Theta,G,f}(X)$ has a repeated factor, there 
always exists a suitable Tschirnhausen transformation $\hat{f}$ of $f$ (cf. \S 3) over $M$ 
(resp. $X-\hat{\Theta}$ of $X-\Theta$ over $k$) such that $\mathcal{RP}_{\Theta,G,\hat{f}}(X)$ 
(resp. $\mathcal{RP}_{\hat{\Theta},G,f}(X)$) has no repeated factors (cf. \cite{Gir83}, 
\cite[Alg. 6.3.4]{Coh93}, \cite{Col95}). 
\end{remark}

In the case where $\mathcal{RP}_{\Theta,G,f}(X)$ has no repeated factors, 
we have the following theorem: 

\begin{theorem}
For $H\leq G\leq S_m$, let $\Theta$ be a $G$-primitive $H$-invariant. 
We assume $\mathrm{Gal}(f)\leq G$ and that $\mathcal{RP}_{\Theta,G,f}(X)$ has no repeated factors. 
Then the following two assertions hold\,{\rm :}\\
{\rm (i)} For $\pi\in G$, the fixed group of the field $M\bigl(\omega_{f}(\pi(\Theta))\bigr)$ 
corresponds to $\mathrm{Gal}(f)\cap \pi H\pi^{-1}$. 
In particular, the fixed group of $\mathrm{Spl}_M \mathcal{RP}_{\Theta,G,f}(X)$ corresponds to 
$\mathrm{Gal}(f)\cap \bigcap_{\pi\in G}\pi H\pi^{-1}$\,{\rm ;} \\
{\rm (ii)} let $\varphi : G\rightarrow S_{[G:H]}$ denote the permutation representation of 
$G$ on the left cosets of $G/H$ given by the left multiplication. Then we 
have a realization of the Galois group of $\mathrm{Spl}_M \mathcal{RP}_{\Theta,G,f}(X)$ 
as a subgroup of $S_{[G:H]}$ by $\varphi(\mathrm{Gal}(f))$. 
\end{theorem}
\section{Formal Tschirnhausen transformation}\label{seTschin}

We recall the geometric interpretation of Tschirnhausen transformations 
which is given in \cite{HM}. 
Let $f(X)$ be monic separable polynomial of degree $n$ in $M[X]$ 
with a fixed order of the roots $\alpha_1,\ldots,\alpha_n$ of $f(X)$ in $\overline{M}$. 
A Tschirnhausen transformation of $f(X)$ over $M$ is a polynomial of the form 
\[
g(X)=\prod_{i=1}^n 
\bigl(X-(c_0+c_1\alpha_i+\cdots+c_{n-1}\alpha_i^{n-1})\bigr),\ c_j \in M.
\]
Two polynomials $f(X)$ and $g(X)$ in $M[X]$ are Tschirnhausen equivalent over $M$ if they are 
Tschirnhausen transformations over $M$ of each other. 
For two irreducible separable polynomials $f(X)$ and $g(X)$ in $M[X]$, 
$f(X)$ and $g(X)$ are Tschirnhausen equivalent over $M$ if and only if 
the quotient fields $M[X]/(f(X))$ and $M[X]/(g(X))$ are isomorphic over $M$. 

In order to obtain an answer to the field intersection problem of $k$-generic polynomials 
via multi-resolvent polynomials, we first treat a general polynomial 
whose roots are $n$ indeterminates $x_1,\ldots,x_n$: 
\begin{align*}
f_\bs(X)\, &=\, \prod_{i=1}^n(X-x_i)\, =\, X^n-s_1X^{n-1}+s_2X^{n-2}+\cdots+(-1)^n s_n\ 
\in k[\bs][X]
\end{align*} 
where $k[x_1,\ldots,x_n]^{S_n}=k[\bs]:=k[s_1,\ldots,s_n], \bs=(s_1, \ldots, s_n),$ and $s_i$ 
is the $i$-th elementary symmetric function in $n$ variables $\bx=(x_1,\ldots,x_n)$. 

Let $R_\bx:=k[x_1,\ldots,x_n]$ and $R_\by:=k[y_1,\ldots,y_n]$ be polynomial rings over $k$. 
Put $R_{\bx,\by}:=k[\bx,\by,1/\Delta_\bx,1/\Delta_\by]$ where 
$\Delta_\bx:=\prod_{1\leq i<j\leq n}(x_j-x_i)$ and 
$\Delta_\by:=\prod_{1\leq i<j\leq n}(y_j-y_i)$. 
We define the interchanging involution $\iota_{\bx,\by}$ which exchanges the 
indeterminates $x_i$'s and the $y_i$'s: 
\begin{align}
\iota_{\bx,\by}\ :\ R_{\bx,\by}\longrightarrow R_{\bx,\by},\ 
x_i\longmapsto y_i,\ y_i\longmapsto x_i,\quad (i=1,\ldots,n).\label{defiota}
\end{align}
We take another general polynomial $f_\bt(X):=\iota_{\bx,\by}(f_\bs(X))\in k[\bt][X], 
\bt=(t_1,\ldots,t_n)$ with roots $y_1,\ldots,y_n$ where $t_i=\iota_{\bx,\by}(s_i)$ 
is the $i$-th elementary symmetric function in $\by=(y_1,\ldots,y_n)$. 
We put 
\[
K\ :=\ k(\bs,\bt);
\]
it is regarded as the rational function field over $k$ with $2n$ variables. 
We put $f_{\bs,\bt}(X):=f_\bs(X)f_\bt(X)$. 
The polynomial $f_{\bs,\bt}(X)$ of degree $2n$ is defined over $K$. We denote 
\begin{align*}
\Gs\, :=\, \mathrm{Gal}(f_\bs/K),\quad \Gt\, :=\, \mathrm{Gal}(f_\bt/K),\quad 
\Gst\, :=\, \mathrm{Gal}(f_{\bs,\bt}/K). 
\end{align*}
Then we have $\Gst=\Gs\times\Gt, \Gs\cong \Gt\cong S_n$ and $k(\bx,\by)^{\Gst}=K$. 

We intend to apply the results of the previous section for 
$m=2n, G=\Gst\leq S_{2n}$ and $f=f_{\bs,\bt}$. 

Note that over the field $\mathrm{Spl}_K f_{\bs,\bt}(X)=k(\bx,\by)$, there exist $n!$ 
Tschirnhausen transformations from $f_\bs(X)$ to $f_\bt(X)$ with respect to 
$y_{\pi(1)},\ldots,y_{\pi(n)}$ for $\pi\in S_n$. 
We study the field of definition of each Tschirnhausen transformation from 
$f_\bs(X)$ to $f_\bt(X)$. 
Let 
\begin{align*}
D:=
\left(
\begin{array}{ccccc}
1 & x_1 & x_1^2 & \cdots & x_1^{n-1}\\ 
1 & x_2 & x_2^2 & \cdots & x_2^{n-1}\\
\vdots & \vdots & \vdots & \ddots & \vdots\\
1 & x_n & x_n^2 & \cdots & x_n^{n-1}\end{array}\right)
\end{align*}
be the Vandermonde matrix of size $n$. 
The matrix $D\in M_n(k(\bx))$ is invertible because the determinant of $D$ equals 
${\rm det}\, D=\Delta_\bx$. 
When char $k\neq 2$, the field $k(\bs)(\Delta_\bx)$ is a quadratic extension of $k(\bs)$ 
which corresponds to the fixed field of the alternating group of degree $n$. 
We define the $n$-tuple $(u_0(\mathbf{x},\mathbf{y}),\ldots,
u_{n-1}(\mathbf{x},\mathbf{y}))\in (R_{\bx,\by})^n$ by 
\begin{align}
\left(\begin{array}{c}u_0(\mathbf{x},\mathbf{y})\\ u_1(\mathbf{x},
\mathbf{y})\\ \vdots \\ u_{n-1}(\mathbf{x},\mathbf{y})\end{array}\right)
:=D^{-1}\left(\begin{array}{c}y_1\\ y_2\\ \vdots \\ 
y_n\end{array}\right). \label{defu}
\end{align}
Cramer's rule shows 
\begin{align}
u_i(\mathbf{x},\mathbf{y})=\Delta_\bx^{-1}\cdot\mathrm{det}
\left(\begin{array}{cccccccc}
1 & x_1 & \cdots & x_1^{i-1} & y_1 & x_1^{i+1} & \cdots & x_1^{n-1}\\ 
1 & x_2 & \cdots & x_2^{i-1} & y_2 & x_2^{i+1} & \cdots & x_2^{n-1}\\
\vdots & \vdots & & \vdots & \vdots & \vdots & & \vdots\\
1 & x_n & \cdots & x_n^{i-1} & y_n & x_n^{i+1} & \cdots & x_n^{n-1}
\end{array}\right).\label{cramer}
\end{align}
In order to simplify the presentation, we write 
\[
u_i:=u_i(\mathbf{x},\mathbf{y}),\quad (i=0,\ldots,n-1). 
\]
The Galois group $\Gst$ acts on the orbit $\{\pi(u_i)\ |\ \pi\in \Gst \}$ 
via regular representation from the left. 
However this action is not faithful. 
We put 
\[
\Hst:=\{(\pi_\bx, \pi_\by)\in \Gst\ |\ \pi_\bx(i)=\pi_\by(i)\ \mathrm{for}\ 
i=1,\ldots,n \}\cong S_n. 
\]
If $\pi \in \Hst$ then we have $\pi(u_i)=u_i$ for $i=0,\ldots,n-1$. 
Indeed we see the following lemma: 
\begin{lemma}\label{stabil}
For $i$, $0\leq i\leq n-1$, $u_i$ is a $\Gst$-primitive $\Hst$-invariant. 
\end{lemma}

Let $\Theta:=\Theta(\bx,\by)$ be a $\Gst$-primitive $\Hst$-invariant. 
Let $\opi=\pi\Hst$ be a left coset of $\Hst$ in $\Gst$. 
The group $\Gst$ acts on the set $\{ \pi(\Theta)\ |\ \opi\in \Gst/\Hst\}$ 
transitively from the left through the action on the set $\Gst/\Hst$ of left cosets. 
Each of the sets $\{ \overline{(1,\pi_\by)}\ |\ (1,\pi_\by)\in \Gst\}$ 
and $\{ \overline{(\pi_\bx,1)}\ |\ (\pi_\bx,1)\in \Gst\}$ forms a complete residue 
system of $\Gst/\Hst$, and hence the subgroups $\Gs$ and $\Gt$ of $\Gst$ act on the set 
$\{ \pi(\Theta)\ |\ \opi\in \Gst/\Hst\}$ transitively. 
For $\opi=\overline{(1,\pi_\by)}\in \Gst/\Hst$, we obtain the following equality 
from the definition (\ref{defu}): 
\[
y_{\pi_\by(i)} = \pi_\by(u_0)+\pi_\by(u_1) x_i+\cdots+\pi_\by(u_{n-1})x_i^{n-1}\ 
\mathrm{for}\ i=1,\ldots,n. 
\]
Hence the set $\{(\pi(u_0),\ldots,\pi(u_{n-1}))\ |\ \opi\in \Gst/\Hst\}$ 
gives coefficients of $n!$ different Tschirnhausen transformations from $f_\bs(X)$ 
to $f_\bt(X)$ each of which is defined over $K(\pi(u_0),\ldots,\pi(u_{n-1})),$ respectively. 
We call $K(\pi(u_0),\ldots,\pi(u_{n-1})), (\opi\in \Gst/\Hst)$, a field of formal Tschirnhausen 
coefficients from $f_\bs(X)$ to $f_\bt(X)$. 
We put $v_i:=\iota_{\bx,\by}(u_i)$ for $i=0,\ldots,n-1$. 
Then $v_i$ is also a $\Gst$-primitive $\Hst$-invariant and $K(\pi(v_0),\ldots,\pi(v_{n-1}))$ 
gives a field of formal Tschirnhausen coefficients from $f_\bt(X)$ to $f_\bs(X)$. 
\begin{proposition}\label{prop1}
Let $\Theta$ be a $\Gst$-primitive $\Hst$-invariant. 
Then we have $k(\bx,\by)^{\pi\Hst \pi^{-1}}$ $=$ $K(\pi(u_0),\ldots,\pi(u_{n-1}))$ 
$=$ $K(\pi(\Theta))$ and $[K(\pi(\Theta)) : K]=n!$ for each $\opi\in \Gst/\Hst$. 
\end{proposition}
Hence, for each of the $n!$ fields $K(\pi(\Theta))$, we have 
$\mathrm{Spl}_{K(\pi(\Theta))} f_\bs(X)=\mathrm{Spl}_{K(\pi(\Theta))} f_\bt(X), 
(\opi\in \Gst/\Hst)$. 
We also obtain the following proposition: 
\begin{proposition}\label{propLL}
Let $\Theta$ be a $\Gst$-primitive $\Hst$-invariant. 
Then we have 
\begin{align*}
&{\rm (i)}\ K(\bx)\cap K(\pi(\Theta))=K(\by)\cap K(\pi(\Theta))=K\quad \textrm{for}\quad 
\opi\in \Gst/\Hst\,{\rm ;}\\
&{\rm (ii)}\ K(\bx,\by)=K(\bx,\pi(\Theta))=K(\by,\pi(\Theta))\quad \textrm{for}\quad 
\opi\in \Gst/\Hst\,{\rm ;}\\
&{\rm (iii)}\ K(\bx,\by)=K(\pi(\Theta)\ |\ \opi\in \Gst/\Hst). 
\end{align*}
\end{proposition}

We consider the formal $\Gst$-relative $\Hst$-invariant resolvent polynomial of degree $n!$ 
by $\Theta$: 
\[
\mathcal{RP}_{\Theta,\Gst}(X)=\prod_{\opi\in \Gst/\Hst}(X-\pi(\Theta))\in k(\bs,\bt)[X]. 
\]
It follows from Proposition \ref{prop1} that $\mathcal{RP}_{\Theta,\Gst}(X)$ is irreducible 
over $k(\bs,\bt)$. 
From Proposition \ref{propLL} we have one of the basic results: 
\begin{theorem}\label{th-gen}
The polynomial $\mathcal{RP}_{\Theta,\Gst}(X)$ is $k$-generic for $S_n\times S_n$. 
\end{theorem}

\section{Field intersection problem}\label{seInt}

For $\ba=(a_1,\ldots,a_n), \bb=(b_1,\ldots,b_n)\in M^n$, we take some fixed order of the roots 
$\alpha_1,\ldots,\alpha_n$ (resp. $\beta_1,\ldots,\beta_n$) of $f_\ba(X)$ (resp. $f_\bb(X)$) 
in $\overline{M}$. 
Put $f_{\ba,\bb}(X):=f_\ba(X)f_\bb(X)\in M[X]$. 
We denote $L_{\ba} := M(\alpha_1,\ldots,\alpha_n)$ and 
$L_{\bb} := M(\beta_1,\ldots,\beta_n)$; then 
$L_{\ba} = \mathrm{Spl}_{M} f_\ba(X)$, $L_{\bb} = \mathrm{Spl}_{M} f_\bb(X)$ and 
$L_{\ba}\,L_{\bb} = \mathrm{Spl}_{M} f_{\ba,\bb}(X)$. 
We define a specialization homomorphism 
$\omega_{f_{\ba,\bb}}$ by 
\begin{align*}
\omega_{f_{\ba,\bb}} 
: R_{\bx,\by} &\longrightarrow M(\alpha_1,\ldots,\alpha_n,\beta_1,\ldots,\beta_n) 
= L_{\ba}\,L_{\bb},\\
\Theta(\bx,\by) &\longmapsto\Theta(\alpha_1,\ldots,\alpha_n,\beta_1,\ldots,\beta_n). 
\end{align*}
Denote $\Delta_\ba:=\omega_{f_{\ba,\bb}}(\Delta_\bx)$ and 
$\Delta_\bb:=\omega_{f_{\ba,\bb}}(\Delta_\by)$. 
We assume that both of the polynomials $f_\ba(X)$ and $f_\bb(X)$ are separable over $M$, 
i.e. $\omega_{f_{\ba,\bb}}(\Delta_\bx)\cdot\omega_{f_{\ba,\bb}}(\Delta_\by)\neq 0$. 
Put 
\[
G_\ba:=\mathrm{Gal}(f_{\ba}/M),\quad G_\bb:=\mathrm{Gal}(f_{\bb}/M),\quad 
G_{\ba,\bb}:=\mathrm{Gal}(f_{\ba,\bb}/M).
\]
Then we may naturally regard $G_{\ba,\bb}$ as a subgroup of $\Gst$. 
For $\opi \in \Gst/\Hst$, we put $c_{i,\pi}:=\omega_{f_{\ba,\bb}}(\pi(u_i)), 
d_{i,\pi}:=\omega_{f_{\ba,\bb}}\bigl(\pi(\iota_{\bx,\by}(u_i))\bigr)$, ($i=0,\ldots,n-1$). 
Then we have 
\begin{align}
\beta_{\pi_\by(i)}\,&=\, c_{0,\pi} + c_{1,\pi}\,\alpha_{\pi_\bx(i)}
+ \cdots + c_{n-1,\pi}\,\alpha_{\pi_\bx(i)}^{n-1}, \label{b=a's}\\
\alpha_{\pi_\bx(i)}\,&=\, d_{0,\pi} + d_{1,\pi}\,\beta_{\pi_\by(i)}
+ \cdots + d_{n-1,\pi}\,\beta_{\pi_\by(i)}^{n-1} \label{a=b's}
\end{align}
for each $i = 1, \ldots, n$. 

For each $\opi\in\Gst/\Hst$, there exists a Tschirnhausen transformation from $f_\ba(X)$ 
to $f_\bb(X)$ over its field of Tschirnhausen coefficients $M(c_{0,\pi},\ldots,c_{n-1,\pi})$; 
the $n$-tuple $(d_{0,\pi},\ldots,d_{n-1,\pi})$ gives the coefficients of a transformation 
of the inverse direction. 
From the assumption $\Delta_\ba\cdot \Delta_\bb\neq 0$, we see the following lemma 
(cf. \cite[p. 141]{JLY02}, \cite{HM}): 
\begin{lemma}\label{lemM}
Let $M'/M$ be a field extension. 
If $f_\bb(X)$ is a Tschirnhausen transformation of $f_\ba(X)$ over $M'$, then $f_\ba(X)$ 
is a Tschirnhausen transformation of $f_\bb(X)$ over $M'$. 
In particular, we have $M(c_{0,\pi},\ldots,c_{n-1,\pi})=M(d_{0,\pi},\ldots,d_{n-1,\pi})$ 
for every $\opi\in\Gst/\Hst$. 
\end{lemma}

In order to obtain an answer to the field intersection problem of $f_\bs(X)$ we study the 
$n!$ fields $M(c_{0,\pi},\ldots,c_{n-1,\pi})$ of Tschirnhausen coefficients from $f_\ba(X)$ 
to $f_\bb(X)$ over $M$. 
\begin{proposition}\label{propc}
Under the assumption $\Delta_\ba\cdot\Delta_\bb\neq 0$, we have 
the following two assertions\,{\rm :} 
\begin{align*}
&{\rm (i)}\ \ \mathrm{Spl}_{M(c_{0,\pi},\ldots,c_{n-1,\pi})} f_\ba(X)
=\mathrm{Spl}_{M(c_{0,\pi},\ldots,c_{n-1,\pi})} f_\bb(X)\, \ 
\textit{for each}\ \,\opi\in\Gst/\Hst\, 
{\rm ;}\\
&{\rm (ii)}\ L_\ba L_\bb=L_\ba\, M(c_{0,\pi},\ldots,c_{n-1,\pi})=L_\bb\, 
M(c_{0,\pi},\ldots,c_{n-1,\pi})\, \ \textit{for each}\ \, \opi\in\Gst/\Hst.
\end{align*}
\end{proposition}

Let $\Theta=\Theta(\bx,\by)$ be a $\Gst$-primitive $\Hst$-invariant. 
Applying the specialization $\omega_{f_{\ba,\bb}}$ to $\Theta$, we have a $\Gst$-relative 
$\Hst$-invariant resolvent polynomial of $f_{\ba,\bb}$ by $\Theta$: 
\begin{align*}
\mathcal{RP}_{\Theta,\Gst,f_{\ba,\bb}}(X)\, =\, 
\prod_{\opi\in \Gst/\Hst} \bigl(X-\omega_{f_{\ba,\bb}}(\pi(\Theta))\bigr)\in M[X]. 
\end{align*}
The resolvent polynomial $\mathcal{RP}_{\Theta,\Gst,f_{\ba,\bb}}(X)$ is also called 
(absolute) multi-resolvent (cf. \cite{GLV88}, \cite{RV99}, \cite{Val}, \cite{Ren04}). 
\begin{proposition}\label{prop12}
For $\ba,\bb \in M^n$ with $\Delta_\ba\cdot\Delta_\bb\neq 0$, suppose that the resolvent 
polynomial $\mathcal{RP}_{\Theta,\Gst,f_{\ba,\bb}}(X)$ has no repeated factors. 
Then the following two assertions hold\,{\rm :}\\
$(\mathrm{i})$\ $M(c_{0,\pi},\ldots,c_{n-1,\pi})=M\bigl(\omega_{f_{\ba,\bb}}(\pi(\Theta))\bigr)$ for each 
$\opi\in\Gst/\Hst$\,{\rm ;}\\
$(\mathrm{ii})$\ ${\rm Spl}_M f_{\ba,\bb}(X)
=M(\omega_{f_{\ba,\bb}}(\pi(\Theta))\ |\ \opi\in\Gst/\Hst)$. 
\end{proposition}

\begin{definition}
For a separable polynomial $f(X)\in k[X]$ of degree $d$, the decomposition type of $f(X)$ 
over $M$, denoted by {\rm DT}$(f/M)$, is defined as the partition of $d$ induced by the 
degrees of the irreducible factors of $f(X)$ over $M$. 
We define the decomposition type {\rm DT}$(\mathcal{RP}_{\Theta,G,f}/M)$ of 
$\mathcal{RP}_{\Theta,G,f}(X)$ over $M$ by {\rm DT}$(\mathcal{RP}_{\Theta,G,\hat{f}}/M)$ where 
$\hat{f}(X)$ is a Tschirnhausen transformation of $f(X)$ over $M$ which satisfies that 
$\mathcal{RP}_{\Theta,G,\hat{f}}(X)$ has no repeated factors (cf. Remark \ref{remGir}). 
\end{definition}

We write $\mathrm{DT}(f):=\mathrm{DT}(f/M)$ for simplicity. 
From Theorem \ref{thfun}, the decomposition type 
$\mathrm{DT}(\mathcal{RP}_{\Theta,\Gst,f_{\ba,\bb}})$ coincides with the partition of $n!$ 
induced by the lengths of the orbits of $\Gst/\Hst$ under the action of 
$\mathrm{Gal}(f_{\ba,\bb})$. 

Hence, by Proposition \ref{prop12}, $\mathrm{DT}(\mathcal{RP}_{\Theta,\Gst,f_{\ba,\bb}})$ gives 
the degrees of $n!$ fields of Tschirnhausen coefficients $M(c_{0,\pi},\ldots,c_{n-1,\pi})$ from 
$f_\ba(X)$ to $f_\bb(X)$ over $M$; the degree of $M(c_{0,\pi},\ldots,c_{n-1,\pi})$ over $M$ is 
equal to $|\mathrm{Gal}(f_{\ba,\bb})|/|\mathrm{Gal}(f_{\ba,\bb})\cap\pi\Hst\pi^{-1}|$. 

We conclude that the decomposition type of the resolvent polynomial 
$\mathcal{RP}_{\Theta,\Gst,f_{\ba,\bb}}(X)$ over $M$ gives us information 
about the field intersection problem of $f_\bs(X)$ through the degrees of the fields 
of Tschirnhausen coefficients $M(c_{0,\pi},\ldots,c_{n-1,\pi})$ over $M$ 
which is determined by the degeneration of the Galois group $\mathrm{Gal}(f_{\ba,\bb})$ 
under the specialization $(\bs, \bt) \mapsto (\ba, \bb)$. 

\begin{theorem}\label{throotf}
Let $\Theta$ be a $\Gst$-primitive $\Hst$-invariant. 
For $\ba,\bb \in M^n$ with $\Delta_\ba\cdot\Delta_\bb\neq 0$, the following three 
conditions are equivalent\,{\rm :}\\
{\rm (1)} $M[X]/(f_\ba(X))$ and $M[X]/(f_\bb(X))$ are $M$-isomorphic\,{\rm ;}\\
{\rm (2)} There exists $\pi\in \Gst$ such that $\omega_{f_{\ba,\bb}}(\pi(\Theta))\in M$\,{\rm ;}\\
{\rm (3)} The decomposition type ${\rm  DT}(\mathcal{RP}_{\Theta,\Gst,f_{\ba,\bb}})$ over 
$M$ includes $1$.
\end{theorem}

In the case where $G_\ba$ and $G_\bb$ are isomorphic to a transitive subgroup $G$ of $S_n$ 
and every subgroups of $G$ with index $n$ are conjugate in $G$, the condition that 
$M[X]/(f_\ba(X))$ and $M[X]/(f_\bb(X))$ are $M$-isomorphic is equivalent to the condition 
that $\mathrm{Spl}_M f_\ba(X)$ and $\mathrm{Spl}_M f_\bb(X)$ coincide. 
Hence we obtain an answer to the field isomorphism problem via 
the resolvent polynomial $\mathcal{RP}_{\Theta,\Gst,f_{\ba,\bb}}(X)$. 
\begin{corollary}[The field isomorphism problem]\label{cor1}
For $\ba,\bb \in M^n$ with $\Delta_\ba\cdot\Delta_\bb\neq 0$, we assume that 
both of $f_\ba(X)$ and $f_\bb(X)$ are irreducible over $M$, that $G_\ba$ and $G_\bb$ are 
isomorphic to $G$ and that all subgroups of $G$ with index $n$ are conjugate in $G$. 
Then ${\rm  DT}(\mathcal{RP}_{\Theta,\Gst,f_{\ba,\bb}})$ includes $1$ if and only if 
$\mathrm{Spl}_M f_\ba(X)$ and $\mathrm{Spl}_M f_\bb(X)$ coincide. 
\end{corollary}
\begin{remark}
If $G$ is one of the symmetric group $S_n$ of degree $n$, ($n\neq 6$), the alternating 
group of degree $n$, ($n\neq 6$), and solvable transitive subgroups of $S_p$ of prime degree 
$p$, then all subgroups of $G$ with index $n$ or $p$, respectively, are conjugate in $G$ 
(cf. \cite{Hup67}, \cite{BJY86}).
\end{remark}
\begin{example}
In the case where $G\leq S_n$ has $r$ conjugacy classes of subgroups of index $n$, 
we get an answer to the field isomorphism problem by applying Theorem \ref{throotf} repeatedly. 

For example, when char $k\neq 2$, the polynomials 
$f_{\bs}(X) := f_{s,t}^{D_4}(X)=X^4+sX^2+t$ and 
$g_{\bs}(X) := g_{s,t}^{D_4}(X)=X^4+2sX^2+(s^2-4t)$, $\bs=(s,t)$, are $k$-generic for $D_4$ 
and have the same splitting field over $k(s,t)$. 
However their root fields are not isomorphic over $k(s,t)$. 

For $\ba=(a,b)$, $\ba'=(a',b')\in M^2$ with $G_\ba=G_{\ba'}=D_4$, we see that 
$\mathrm{Spl}_M f_\ba(X)=\mathrm{Spl}_M f_{\ba'}(X)$ if and only if either 
$M[X]/(f_\ba(X))\cong_M M[X]/(f_{\ba'}(X))$ or $M[X]/(f_\ba(X))\cong_M M[X]/(g_{\ba'}(X))$. 
Hence, by applying Theorem \ref{throotf} twice, we obtain an answer to the field isomorphism 
problem. 

The decomposition types of the corresponding multi-resolvent polynomials 
$\mathcal{RP}_{\Theta,\Gst,f_\ba f_{\ba'}}(X)$ and $\mathcal{RP}_{\Theta,\Gst,f_\ba g_{\ba'}}(X)$ 
are given as $8,4,4,2,2,2,1,1$ and $8,8,4,2,2$. 
Note, in this case, that the latter decomposition type also means the isomorphism of the two 
splitting fields of $f_\ba(X)$ and of $f_{\ba'}(X)$ over $M$ although it does not include $1$ 
(cf. \cite{HM-2}). 
\end{example}

\section{generic polynomial for $H_1\times H_2$}\label{seGenPoly}
Let $H_1$ and $H_2$ be subgroups of $S_n$. 
As an analogue to Theorem \ref{th-gen}, we obtain a $k$-generic polynomial 
for $H_1\times H_2$, the direct product of groups $H_1$ and $H_2$. 
\begin{theorem}\label{thgen}
Let $M=k(q_1,\ldots,q_l,r_1,\ldots,r_m)$, $(1\leq l,\, m\leq n-1)$ be the rational function 
field over $k$ with $(l+m)$ variables. 
For $\ba\in {k(q_1,\ldots,q_l)}^n, \bb\in {k(r_1,\ldots,r_m)}^n$, we assume that 
$f_\ba(X)\in M[X]$ and $f_\bb(X)\in M[X]$ be $k$-generic polynomials for $H_1$ and $H_2$, 
respectively. 
If $\mathcal{RP}_{\Theta,\Gst,f_{\ba,\bb}}(X)\in M[X]$ has no repeated factors, then 
$\mathcal{RP}_{\Theta,\Gst,f_{\ba,\bb}}(X)$ is a $k$-generic polynomial 
for $H_1\times H_2$ which is not necessary irreducible. 
\end{theorem}

\begin{example} 
In each Tschirnhausen equivalence class, we are always able to choose a specialization 
$\bs\mapsto \ba\in M^n$ of the polynomial $f_\bs(X)$ which satisfy $a_1=0$ and $a_{n-1}=a_{n}$ 
(see \cite[\S 8.2]{JLY02}). 
Thus the polynomial 
\begin{align*}
X^n+q_2X^{n-2}+\cdots+q_{n-2}X^2+q_{n-1}X+q_{n-1}
\end{align*}
is $k$-generic for $S_n$ with $(n-2)$ parameters $q_2,\ldots,q_{n-1}$ over an arbitrary field $k$. 
For $M=k(q_2,\ldots,q_{n-1}, r_2,\ldots,r_{n-1})$, we take 
$\ba=(0,q_2,\ldots,q_{n-1},q_{n-1})\in M^n, \bb=(0,r_2,\ldots,r_{n-1},r_{n-1})\in M^n$. 
While the polynomial $f_{\ba}(X)f_{\bb}(X)$ of degree $2n$ is $k$-generic for 
$S_n\times S_n$, the resolvent polynomial $\mathcal{RP}_{\Theta,\Gst,f_{\ba,\bb}}(X)$ 
(with no repeated factors) realizes an irreducible $k$-generic polynomial for 
$S_n\times S_n$ of degree $n!$. 
\end{example}

\begin{example}
In the case of $n=3$, some explicit examples of sextic $k$-generic polynomials 
$f_{s,t}^{H_1\times H_2}(X)$ for transitive subgroups $H_1\times H_2$ of $S_6$ are 
given in \cite{HM}. 
We give another examples by taking $\Theta=x_1y_1+x_2y_2+x_3y_3$, 
$f_s^{S_3}(X)=X^3+sX+s$, $f_s^{C_3}(X)=X^3-sX^2-(s+3)X-1$, $f_s^{C_2}(X)=X(X^2-s)$, 
$f_s^{\{1\}}(X)=X(X^2-1)$ when char $k\neq 3$. 
Then we get the following $k$-generic polynomials 
$h^{H_1,H_2}(X):=\mathcal{RP}_{\Theta,\Gst,f_s^{H_1}f_t^{H_2}}(X)$ for $H_1\times H_2$: 
\begin{align*}
h^{S_3,S_3}(X)\,&=\, 
X^6  - 6stX^4 - 27stX^3 + 9s^2t^2X^2 + 81s^2t^2X -s^2t^2(4st+27s + 27t),\\
h^{S_3,C_3}(X)\,&=\,
X^6 + 2s(t^2+3t+9)X^4+  s(2t+3)(t^2+3t+9)X^3+ s^2(t^2+3t+9)^2X^2\\  
&\hspace*{11mm}+ s^2(2t+3)(t^2+3t+9)^2X +s^2(t^2+3t+9)^2(t^2+3t+s+9),\\
h^{S_3,C_2}(X)\,&=\,X^6+ 6stX^4 + 9s^2t^2X^2 +s^2t^3(4s+27),\\
h^{S_3,\{1\}}(X)\,&=\,X^6 + 6sX^4 + 9s^2X^2+s^2(4s+27),\\
h^{C_3,C_2}(X)\,&=\,X^6-2t(s^2+3s+9)X^4+t^2(s^2+3s+9)^2X^2-t^3(s^2+3s+9)^2.
\end{align*}
\end{example}

\section{Solvable quintic generic polynomial}\label{seQuin}

We recall some solvable quintic generic polynomials (cf. \cite{Lec98}, \cite{JLY02}, \cite{HT03}). 
Let $\sigma:=(12345)$, $\rho:=(1243)$, $\tau:=\rho^2$, $\omega:=(12)\in S_5$ acting on 
$k(x_1,\ldots,x_5)$ by $\pi(x_i)=x_{\pi(i)}, (\pi \in S_5)$. 
For simplicity, in this section, we write 
\[
C_5=\langle\sigma\rangle,\quad D_5=\langle\sigma,\tau\rangle,\quad 
F_{20}=\langle\sigma,\rho\rangle,\quad S_5=\langle\sigma,\omega\rangle,
\] 
where $C_5$ (resp. $D_5, F_{20}, S_5$) is the cyclic (resp. dihedral, Frobenius, symmetric) 
group of order $5$ (resp. $10$, $20$, $120$). 
Put 
\begin{align}
x:=\biggl(\frac{x_1-x_4}{x_1-x_3}\biggr)
\bigg{/}\biggl(\frac{x_2-x_4}{x_2-x_3}\biggr),\quad 
y:=\biggl(\frac{x_2-x_5}{x_2-x_4}\biggr)
\bigg{/}\biggl(\frac{x_3-x_5}{x_3-x_4}\biggr). \label{defxy}
\end{align}
Then the symmetric group $S_5$ of degree $5$ faithfully acts on $k(x,y)$ in the manner, 
\begin{align}
\sigma\ &:\ x\longmapsto y,\quad y\longmapsto -\frac{y-1}{x},&
\tau\ &:\ x\longmapsto x,\quad y\longmapsto -\frac{x-1}{y},\label{actxy}\\
\rho\ &:\ x\longmapsto \frac{x}{x-1},\quad y\longmapsto \frac{y-1}{x+y-1},&
\omega\ &:\ x\longmapsto \frac{1}{x},\quad y\longmapsto \frac{x+y-1}{x}.\nonumber
\end{align}
We take a $D_5$-primitive $\langle\tau\rangle$-invariant $x$ and have 
\[
\{ \pi(x)\ |\ \opi \in D_5/\langle\tau\rangle\}
=\Big\{x,y,-\frac{y-1}{x},\frac{x+y-1}{xy},-\frac{x-1}{y}\Big\}. 
\]
Hence we obtain the formal $D_5$-relative $\langle\tau\rangle$-invariant resolvent polynomial 
by $x$, 
\begin{align*}
f_{s,t}^{D_5}(X)&:=\mathcal{RP}_{x,D_5}(X)=\prod_{\opi \in D_5/\langle\tau\rangle}(X-\pi(x))\\
&=X^5+(t-3)X^4+(s-t+3)X^3+(t^2-t-2s-1)X^2+sX+t\ \in\, k(s,t)[X]
\end{align*}
where
\begin{align}
t&:=-\frac{(x-1)(y-1)(x+y-1)}{xy},\nonumber\\
s&:=\sum_{i=0}^4\sigma^i\bigl((x-1)(y-1)\bigr)\label{defst}\\
&\,=-(x-2x^2+x^3+y-4xy+5x^2y-3x^3y+x^4y-2y^2+5xy^2\nonumber\\
&\hspace*{9mm} -5x^2y^2+2x^3y^2+y^3-3xy^3+2x^2y^3-x^3y^3+xy^4)/(x^2y^2).\nonumber
\end{align}
Note that $k(x,y)^{D_5}=k(s,t)$. 
By the normal basis theorem and Remark \ref{remGir}, we see that 
the polynomial $f_{s,t}^{D_5}(X)\in k(s,t)[X]$ is a $k$-generic polynomial for $D_5$ 
(cf. \cite[p. 45]{JLY02}). 
The polynomial $f_{s,t}^{D_5}(X)$ is known as Brumer's quintic. 
Put 
\begin{align*}
d&:=\prod_{\opi \in D_5/\langle\tau\rangle}(\pi(x)-\pi^2(x))\\
&\ =\frac{(x-y)(x+xy-1)(y+xy-1)(x^2+y-1)(x+y^2-1)}{x^3y^3}; 
\end{align*}
then $d$ satisfies the relation 
\begin{align}
d^2=\delta_{s,t}:=s^2-4s^3+4t-14st-30s^2t-91t^2-34st^2+s^2t^2+40t^3+24st^3+4t^4-4t^5. 
\label{surfd}
\end{align}
We note that the discriminant of $f_{s,t}^{D_5}(X)$ with respect to $X$ is given by 
$t^2\delta_{s,t}^2$. 
In the case of char $k\neq 2$, we also see $k(x,y)^{C_5}=k(s,t)(d)$; the 
field $k(s,t,d)$ is a quadratic subextension of $k(x,y)$ over $k(s,t)$. 
By blowing up the surface (\ref{surfd}), Hashimoto-Tsunogai \cite{HT03} 
showed that the fixed field $k(x,y)^{C_5}=k(s,t,d)$ is purely transcendental over $k$. 
A minimal basis of $k(x,y)^{C_5}=k(A,B)$ over $k$ is given explicitly by 
\begin{align*}
A=\frac{s+13t-7st-2t^2+2t^3}{-2+7s+33t+st-8t^2},\quad B=\frac{d}{-2+7s+33t+st-8t^2}. 
\end{align*}
We also see 
\begin{align}
s=\frac{2A+13t-33At-2t^2+8At^2+2t^3}{-1+7A+7t+At},\quad 
d=\frac{2B(-1-11t+t^2)^2}{-1+7A+7t+At}\label{dsABt}
\end{align}
and
\begin{align}
t=-\frac{A^2+A^3-B^2+7AB^2}{1-A+7B^2+AB^2}. \label{tAB}
\end{align}
Hence we obtain the generating polynomial of the field $k(x,y)$ over $k(x,y)^{C_5}=k(A,B)$: 
\begin{align}
f_{A,B}^{C_5}(X):=f_{s,t}^{D_5}(X)\in k(A,B)[X]\label{deffAB}
\end{align}
where $s, t\in k(A,B)$ are given by the above formulas (\ref{dsABt}) and (\ref{tAB}). 
The polynomial $f_{A,B}^{C_5}(X)$ is $k$-generic for $C_5$ with independent parameters $A,B$ 
when char $k\neq 2$. 
The discriminant of $f_{A,B}^{C_5}(X)$ with respect to $X$ is given by
\[
\frac{16B^4(A^2+A^3-B^2+7AB^2)^2P^8}{Q^{14}}
\]
where 
\begin{align}
P=(A^2-A-1)^2+25(A^2+1)B^2+125B^4,\quad Q=1-A+7B^2+AB^2. \label{eqPQ}
\end{align}
We also get an alternative presentation of the $k$-generic polynomial $f_{A,B}^{C_5}(X)$ as 
\begin{align*}
g_{s,t}^{C_5}(X)&:=\mathcal{RP}_{x-y,C_5}(X)\\
&\ =X^5-(2-3s-2t+t^2)X^3+dX^2+(1-3s-10t-4st+3t^2+t^3)X-d.
\end{align*}
By using $s,t,d\in k(A,B)$ in (\ref{dsABt}) and (\ref{tAB}), we have the following 
$k$-generic polynomial (cf. \cite{HT03}): 
\begin{align*}
h_{A,B}^{C_5}(X)&:=g_{s,t}^{C_5}(X)\\
&=X^5-\frac{P}{Q^2}(A^2-2A+15B^2+2)X^3+\frac{P^2}{Q^3}(2BX^2-(A-1)X-2B)
\end{align*}
where $P,Q\in k(A,B)$ are given as in (\ref{eqPQ}) above. 

We note that two polynomials $f_{A,B}^{C_5}(X)$ and $h_{A,B}^{C_5}(X)$ have the same 
splitting field $k(x,y)$ over $k(A,B)$. 
The actions of $\rho$ and of $\tau=\rho^2$ on the fields $k(x,y)^{C_5}=k(s,t,d)=k(A,B)$ and 
$k(x,y)^{D_5}=k(s,t)$ are given by 
\begin{align}
\rho\ &:\ s\longmapsto \frac{s+5t}{t^2},\quad t\longmapsto -\frac{1}{t},\quad 
d\longmapsto \frac{d}{t^3},\quad A\longmapsto -\frac{1}{A},\quad B\longmapsto -\frac{B}{A},
\label{actrho}\\
\tau\ &:\ s\longmapsto s,\quad t\longmapsto t,\quad d\longmapsto -d,\quad 
A\longmapsto A,\quad B\longmapsto -B.\nonumber
\end{align}
\begin{proposition}\label{propisom}
Assume that char $k\neq 2$.\\
{\rm (1)} The polynomials $h_{A,B}^{C_5}(X)$, $h_{-1/A,-B/A}^{C_5}(X)$, $h_{A,-B}^{C_5}(X)$ 
and $h_{-1/A,B/A}^{C_5}(X)$ have the same splitting field $k(x,y)$ over $k(A,B)$.\\
{\rm (2)} The polynomials $f_{s,t}^{D_5}(X)$ and $f_{(s+5t)/t^2,-1/t}^{D_5}(X)$ have 
the same splitting field $k(x,y)$ over $k(s,t)$. 
\end{proposition}
\begin{proof}
(1) Since $f_{\rho^i(A),\rho^i(B)}^{C_5}(X)=\mathcal{RP}_{\rho^i(x),D_5}(X)$, 
$(i=1,2,3)$, each of $\rho^i(x)$, $(i=1,2,3)$ is a 
$D_5$-primitive $\langle\tau\rangle$-invariant. 
Hence the polynomial $f_{A,B}^{C_5}(X)$ and $f_{\rho^i(A),\rho^i(B)}^{C_5}(X)$ have 
the same splitting field $k(x,y)$ over $k(A,B)$. 
The assertion now follows because the splitting fields of $f_{A,B}^{C_5}(X)$ and of 
$h_{A,B}^{C_5}(X)$ over $k(A,B)$ coincide.\\
(2) The assertion follows from $f_{\rho(s),\rho(t)}^{D_5}(X)=\mathcal{RP}_{\rho(x),D_5}(X)$ 
because $\rho(x)$ is a $D_5$-primitive $\langle\tau\rangle$-invariant. 
\end{proof}
\begin{example}\label{exsame}
In Proposition \ref{propisom}, if we specialize the parameter $t:=1$, 
then we see two polynomials 
\begin{align*}
f_s^1(X)&:=f_{s,1}^{D_5}(X)=X^5-2X^4+(s+2)X^3-(2s+1)X^2+sX+1,\\ 
f_s^2(X)&:=f_{s+5,-1}^{D_5}(X)=X^5-4X^4+(s+9)X^3-(2s+9)X^2+(s+5)X-1 
\end{align*} 
have the same splitting field over $k(s)$ including the quadratic field 
$k(s)(\sqrt{47 + 24s + 28s^2 + 4s^3})$ of $k(s)$. 
Now we take a base field $M$ as a number field $K$ and take an algebraic integer $s_1\in K$.  
Note that if $x$ is a root of $f_{s_1,1}^{D_5}(X)$ then $\rho(x)=x/(x-1)$ is a root of 
$f_{s_1+5,-1}^{D_5}(X)$. 
Put 
\begin{align*}
g_{s_1}^1(Y)&:= Y^5\cdot f_{-s_1}^1(1/Y)=Y^5-{s_1}Y^4+(2{s_1}-1)Y^3-({s_1}-2)Y^2-2Y+1,\\
g_{s_1}^2(Y)&:= (-Y)^5\cdot f_{-s_1}^2(-1/Y)=Y^5-({s_1}-5)Y^4-(2{s_1}-9)Y^3-({s_1}-9)Y^2+4Y+1.
\end{align*}
Then we have $g_{s_1}^1(Y)=g_{s_1}^2(Y-1)$; hence, 
if $\theta$ is a root of $g_{s_1}^1(X)$ then $\theta-1$ is a root of $g_{s_1}^2(X)$. 
In particular, both of $\theta$ and $\theta-1$ are units in the same quintic cyclic extension 
$L_5$ of $K$. 
The polynomial $g_{s_1}^1(X)$ is investigated to construct certain parametric systems 
of fundamental units in cyclic quintic fields (cf. \cite{Kih01}, \cite{LPS03}, \cite{Sch06}). 
\end{example}
In the case of char $k=2$, the polynomials $f_{A,B}^{C_5}(X)$ and $h_{A,B}^{C_5}(X)$ 
are not $k$-generic for $C_5$ because $k(x,y)^{D_5}=k(s,t)=k(s,t)(d)$. 
Hence we should choose another generator of the field $k(x,y)^{C_5}$ over 
$k(x,y)^{D_5}=k(s,t)$. 
We take an $S_5$-primitive $C_5$-invariant 
\begin{align*}
e'&:=\sum_{i=0}^4 \sigma^i(xy^2)\\
&\ =xy^2+\frac{y(y+1)^2}{x^2}+\frac{(y+1)(x+y+1)^2}{x^3y^2}+\frac{x^2(x+1)}{y}
+\frac{(x+1)^2(x+y+1)}{xy^3}\\
&\ =\frac{x^2+x^3+x^4+x^5+y+x^4y+y^2+x^2y^2+x^5y^2+x^6y^2+y^3+y^4+xy^4+x^4y^5+xy^6}{x^3y^3}; 
\end{align*}
then we have $k(x,y)^{C_5}=k(s,t)(e')$ and the equality 
\begin{align*}
e'^2 + (s + t + st)e'+ 1  + s + s^3 + t^2 + t^4 + t^5 =0. 
\end{align*}
Thus we put 
\begin{align*}
e&:=\frac{e'}{s+t+st}\\
&\ =\frac{x^2+x^3+x^4+x^5+y+x^4y+y^2+x^2y^2+x^5y^2+x^6y^2+y^3+y^4+xy^4+x^4y^5+xy^6}
{x+x^5+y+x^6y+x^5y^2+x^6y^2+x^5y^4+y^5+x^2y^5+x^4y^5+xy^6+x^2y^6}, 
\end{align*}
and get $k(x,y)^{C_5}=k(s,t)(e)$; the element $e$ satisfies the following equality 
of the Artin-Schreier type: 
\begin{align}
e^2+e+\frac{1+s+s^3+t^2+t^4+t^5}{(s+t+st)^2}=0.\label{eqe}
\end{align}
We note that the actions of $\rho$ and $\tau=\rho^2$ on $k(x,y)^{C_5}=k(s,t,e)$ are given by 
\begin{align*}
\rho\ &:\ s\longmapsto \frac{s+t}{t^2},\quad t\longmapsto \frac{1}{t},\quad 
e\longmapsto e+\frac{1+s+t^2+t^3}{s+t+st},\\
\tau\ &:\ s\longmapsto s,\quad t\longmapsto t,\quad e\longmapsto e+1.  
\end{align*}

For an arbitrary field $k$, by using the action of $\rho$ on $k(s,t)$, 
a $k$-generic polynomial for $F_{20}$ is also obtained as follows. 
The fixed field $k(x,y)^{F_{20}}$ is generated by two elements $\{p,q\}$ 
over $k$ where 
\begin{align}
p:=t-\frac{1}{t},\quad q:=s+\frac{s+5t}{t^2}. \label{defpq}
\end{align}
Hence the fixed field $k(x,y)^{F_{20}}=k(p,q)$ is purely transcendental over $k$. 
The element $(x-1)/x^2$ is an $F_{20}$-primitive $\langle\rho\rangle$-invariant, and we see 
\begin{align*}
&\Bigl\{ \pi\Bigl(\frac{x-1}{x^2}\Bigr)\ \Big{|}\ \opi \in F_{20}/\langle\rho\rangle\Bigr\}\\
&=\Bigl\{\frac{x-1}{x^2}, \frac{y-1}{y^2}, -\frac{x(x+y-1)}{(y-1)^2}, 
-\frac{xy(x-1)(y-1)}{(x+y-1)^2}, -\frac{y(x+y-1)}{(x-1)^2}\Bigr\}.
\end{align*}
Then we obtain the $F_{20}$-relative $\langle\rho\rangle$-invariant resolvent polynomial by 
$(x-1)/x^2$ as 
\begin{align*}
f_{p,q}^{F_{20}}(X)&:=\mathcal{RP}_{(x-1)/x^2,F_{20}}(X)\\
&=X^5+\Bigl(\frac{q^2+5pq-25}{p^2+4}-2p+2\Bigr)X^4\\
&\hspace*{11mm} +\bigl(p^2-p-3q+5\bigr)X^3+(q-3p+8)X^2+(p-6)X+1\ \in\ k(p,q)[X]
\end{align*}
for an arbitrary field $k$, therefore, the polynomial $f_{p,q}^{F_{20}}(X)$ is $k$-generic 
for $F_{20}$. 

When char $k\neq 2$, let us put 
\[
r:=-\frac{5p+2q}{2(p^2+4)}. 
\]
Then from $q=-(5p+8r+2p^2r)/2$ we have $k(p,q)=k(p,r)$. 
Hence we obtain the following $k$-generic polynomial $h_{p,r}^{F_{20}}(X)\in k(p,r)[X]$ 
for $F_{20}$: 
\begin{align*}
g_{p,r}^{F_{20}}(X)&:=f_{p,-(5p+8r+2p^2r)/2}^{F_{20}}(X)\\
&\ =X^5+\Bigl(r^2(p^2+4)-2p-\frac{17}{4}\Bigr)X^4+\Bigl((p^2+4)(3r+1)+\frac{13p}{2}+1\Bigr)X^3\\
&\hspace*{13mm}-\Bigr(r(p^2+4)+\frac{11p}{2}-8\Bigl)X^2+(p-6)X+1.
\end{align*}
The polynomial $g_{p,r}^{F_{20}}(X)$ was constructed by Lecacheux \cite{Lec98}. 

\section{Field intersection problems for solvable quintics}\label{seIntQuin}

The aim of this section is to give an answer to the field intersection problem of 
$k$-generic polynomials $h_{A,B}^{C_5}(X),f_{s,t}^{D_5}(X),f_{p,q}^{F_{20}}(X)$ 
(or $g_{p,r}^{F_{20}}(X)$ when char $k\neq 2$) explicitly via the relative 
(multi-) resolvent polynomials. 

Let $f_{v_1,v_2}^H(X)\in k(v_1,v_2)[X]$ be a quintic $k$-generic polynomial with a solvable 
Galois group $H$, i.e. $H\leq F_{20}$. 
For a fixed polynomial $f_{v_1,v_2}^H(X)$, we write 
\[
L_\ba:=\mathrm{Spl}_M f_\ba^H(X)\quad\mathrm{and}\quad G_\ba:=\mathrm{Gal}(f_\ba^H/M)
\]
for $\ba=(a_1,a_2)\in M^2$. 
Assume that $G_\ba\neq \{1\}$. 
We take the cross-ratios 
\begin{align}
&x:=\xi(x_1,\ldots,x_5)=\biggl(\frac{x_1-x_4}{x_1-x_3}\biggr)
\bigg{/}\biggl(\frac{x_2-x_4}{x_2-x_3}\biggr),\\
&y:=\eta(x_1,\ldots,x_5)=\biggl(\frac{x_2-x_5}{x_2-x_4}\biggr)
\bigg{/}\biggl(\frac{x_3-x_5}{x_3-x_4}\biggr),\nonumber
\end{align}
and $x':=\xi(y_1,\ldots,y_5)$, $y':=\eta(y_1,\ldots,y_5)$ in the same way as (\ref{defxy}) 
in Section \ref{seQuin}. 
For the two fields $k(\bx)=k(x,y)$ and $k(\bx')=k(x',y')$, we take the interchanging involution
\[
\iota \ :\ k(\bx,\bx')\longrightarrow k(\bx,\bx'),\quad
(x,y,x',y')\longmapsto (x',y',x,y)
\]
which is the special case of $\iota_{\bx,\by}$ given by (\ref{defiota}) in Section \ref{seTschin}. 

We take elements $\sigma,\tau,\rho,\omega\in \mathrm{Aut}_k(k(x,y))$ as in the previous section; 
their action on $k(x,y)$ is given by (\ref{actxy}). 
We put $(\sigma',\tau',\rho',\omega'):=(\iota^{-1}\sigma\iota,\iota^{-1}\tau\iota,
\iota^{-1}\rho\iota,\iota^{-1}\omega\iota) \in \mathrm{Aut}_k(k(x',y'))$ and write 
\begin{align*}
C_5&=\langle\sigma\rangle, &D_5&=\langle\sigma,\tau\rangle, &
F_{20}&=\langle\sigma,\rho\rangle, &S_5&=\langle\sigma,\omega\rangle,\\
{C_5}'&=\langle\sigma'\rangle, &{D_5}'&=\langle\sigma',\tau'\rangle, &
{F_{20}}'&=\langle\sigma',\rho'\rangle, &{S_5}'&=\langle\sigma',\omega'\rangle,\\
{C_5}''&=\langle\sigma\sigma'\rangle, &{D_5}''&=\langle\sigma\sigma',\tau\tau'\rangle, &
{F_{20}}''&=\langle\sigma\sigma',\rho\rho'\rangle, &
{S_5}''&=\langle\sigma\sigma',\omega\omega'\rangle. 
\end{align*}

Let $\Theta$ be an $S_5\times {S_5}'$-primitive ${F_{20}}''$-invariant 
and take the formal $S_5\times {S_5}'$-relative (resp. $F_{20}\times {F_{20}}'$-relative) 
${F_{20}}''$-invariant resolvent polynomial of degree $120$ (resp. 20): 
\begin{align*}
\mathcal{R}_{\ba,\ba'}(X)&:=
\mathcal{RP}_{\Theta,S_5\times {S_5}',f_\ba^Hf_{\ba'}^H}(X),\\
\mathcal{R}_{\ba,\ba'}^1(X)&:=
\mathcal{RP}_{\Theta,F_{20}\times {F_{20}}',f_\ba^Hf_{\ba'}^H}(X).
\end{align*}
Since the polynomial $\mathcal{R}_{\ba,\ba'}^1(X)$ divides $\mathcal{R}_{\ba,\ba'}(X)$, we 
put $\mathcal{R}_{\ba,\ba'}^2(X):=\mathcal{R}_{\ba,\ba'}(X)/\mathcal{R}_{\ba,\ba'}^1(X)$. 
Note that we need only the polynomial $\mathcal{R}_{\ba,\ba'}^1(X)$ of degree $20$ instead of 
$\mathcal{R}_{\ba,\ba'}(X)$ of degree $120$ to treat the intersection problem of 
$f_{v_1,v_2}^H(X), (H\leq F_{20})$. 
Indeed we obtain the following theorem: 
\begin{theorem}\label{thF20F20}
For $\ba=(a_1,a_2), \ba'=(a_1',a_2')\in M^2$, assume that 
$f_{\ba}^H(X)$ and $f_{\ba'}^H(X)$ is irreducible over $M$ and $G_\ba\geq G_{\ba'}$. 
The decomposition type of the polynomial $\mathcal{R}_{\ba,\ba'}^1(X)$ over $M$ 
and the Galois group $\mathrm{Gal}(\mathcal{R}_{\ba,\ba'}^1/M)$ give an answer to the field 
intersection problem of $f_{v_1,v_2}^H(X)$ as Table $1$ shows. 
Moreover if $\mathcal{R}^1_{\ba,\ba'}(X)$ has no repeated factors then 
two splitting fields of $f_{\ba}^H(X)$ and of $f_{\ba'}^H(X)$ over $M$ coincide if and only if 
the polynomial $\mathcal{R}_{\ba,\ba'}^1(X)$ has a linear factor over $M$. 
\end{theorem}
\begin{center}
{\rm Table} $1$\vspace*{3mm}\\
{\small 
\begin{tabular}{|c|c|l|l|c|l|l|l|}\hline
$G_\ba$& $G_{\ba'}$ & & GAP ID & $G_{\ba,{\ba'}}$ & & ${\rm DT}(\mathcal{R}_{\ba,\ba'}^1)$ 
& ${\rm DT}(\mathcal{R}_{\ba,\ba'}^2)$ \\ \hline 
& & (I-1) & $[400,205]$ & $F_{20}\times F_{20}$ & $L_\ba\cap L_{\ba'}=M$ 
& $20$ & $100$\\ \cline{3-8} 
& & (I-2) & $[200,42]$ & $(D_5\times D_5)\rtimes C_2$ & $[L_\ba\cap L_{\ba'}:M]=2$ 
& $10^2$ & $50^2$\\ \cline{3-8} 
& $F_{20}$ & (I-3) & $[100,11]$ & $(C_5\times C_5)\rtimes C_4$ & 
$[L_\ba\cap L_{\ba'}:M]=4$ & $5^4$ & $50^2$\\ \cline{3-8}
\raisebox{-1.6ex}[0cm][0cm]{$F_{20}$}
& & (I-4) & $[100,12]$ &$(C_5\times C_5)\rtimes C_4$ & $[L_\ba\cap L_{\ba'}:M]=4$ 
& $10^2$ & $25^4$\\ \cline{3-8} 
& & (I-5) & $[20,3]$ & $F_{20}$ & $L_\ba=L_{\ba'}$ & $5^3,4,1$ & $20^4,10^2$\\ \cline{2-8}
& \raisebox{-1.6ex}[0cm][0cm]{$D_5$} & (I-6) 
& $[200,41]$ & $F_{20}\times D_5$ & $L_\ba\cap L_{\ba'}=M$ & $20$ & $100$\\ \cline{3-8} 
& & (I-7) & $[100,10]$ & $(C_5\times C_5)\rtimes C_4$ & $[L_\ba\cap L_{\ba'}:M]=2$ 
& $20$ & $100$\\ \cline{2-8}
& $C_5$ & (I-8) & $[100,9]$ & $F_{20}\times C_5$ & $L_\ba\cap L_{\ba'}=M$ 
& $20$ & $100$\\ \cline{1-8}
& & (II-1) & $[100,13]$ & $D_5\times D_5$ & $L_\ba\cap L_{\ba'}=M$ & $10^2$ & $50^2$\\ \cline{3-8} 
\raisebox{-1.6ex}[0cm][0cm]{$D_5$} & $D_5$ & (II-2) & $[50,4]$ & $(C_5\times C_5)\rtimes C_2$ 
& $[L_\ba\cap L_{\ba'} : M]=2$ & $5^4$ & $25^4$\\ \cline{3-8}
& & (II-3) & $[10,1]$ & $D_5$ & $L_\ba=L_{\ba'}$ & $5^3,2^2,1$ & $10^8,5^4$\\ \cline{2-8}
& $C_5$ & (II-4) & $[50,3]$ & $D_5\times C_5$ & $L_\ba\cap L_{\ba'}=M$ & $10^2$ 
& $50^2$\\ \cline{1-8}
\raisebox{-1.6ex}[0cm][0cm]{$C_5$} & \raisebox{-1.6ex}[0cm][0cm]{$C_5$} & (III-1) & $[25,2]$ & 
$C_5\times C_5$ & $L_\ba\neq L_{\ba'}$ & $5^4$ & $25^4$\\ \cline{3-8}
& & (III-2) & $[5,1]$ & $C_5$ & $L_\ba=L_{\ba'}$ & $5^3,1^5$ & $5^{20}$\\ \cline{1-8}
\end{tabular}
}\vspace*{5mm}
\end{center}
We checked the decomposition types on Table $1$ using the computer algebra system 
GAP \cite{GAP}. 

It seems, however, complicated to compute the resolvent polynomial 
$\mathcal{R}_{\ba,\ba'}^1(X)$ to display it as an explicit formula. 
This depends on a choice of an $S_5\times S_5'$-primitive $F_{20}''$-invariant $\Theta$ 
and of a minimal basis of the fixed field $K(x,y,x',y')^{F_{20}\times F_{20}'}$ over $K$. 
In Subsections \ref{subR} and \ref{subR2}, we first study the reducible and some tractable cases 
of char $k\neq 2$ and of char $k=2$, respectively. 
We treat in Subsection \ref{subD} the dihedral case and we also evaluate the resolvent 
polynomial $\mathcal{RP}_{P,D_5''}(X)$ for a certain suitable $D_5\times D_5'$-primitive 
$D_5''$-invariant $P$ explicitly. 
This case includes also the cyclic case. 
In Subsection \ref{subC}, we give some numerical examples of Hashimoto-Tsunogai's 
cyclic quintic and Lehmer's simplest quintic. 
Finally in Subsection \ref{subF} we give an answer to the field intersection problem in the 
case of $F_{20}$. 
We do not need to make a formula of another resolvent polynomial because we can use 
the resolvent $\mathcal{RP}_{P,D_5''}(X)$ given in the dihedral case. 
We note that the solution of each case is obtained as certain conditions within the base field $M$. 
\subsection{Reducible and tractable cases of char $k\neq 2$}\label{subR}
Throughout this subsection, we assume that char $k\neq 2$. 
Let $f_{v_1,v_2}^H(X)\in k(v_1,v_2)[X]$ be a $k$-generic polynomial for $H$. 
For the fixed $f_{v_1,v_2}^H(X)$, 
we write $L_\ba:=\mathrm{Spl}_M f_\ba^H(X)$ and $G_\ba:=\mathrm{Gal}(f_\ba^H/M)$ for 
$\ba=(a_1,a_2)\in M^2$. 
We always assume that $f_\ba^H(X)$ has no repeated factors over $M$ and $G_\ba\neq \{1\}$. 
In this subsection, we treat the case where $f_\ba^H(X)$ is reducible over $M$. 
First we take Brumer's quintic, 
\begin{align*}
f_{s,t}^{D_5}(X)=X^5+(t-3)X^4+(s-t+3)X^3+(t^2-t-2s-1)X^2+sX+t\ \in\ k(s,t)[X].
\end{align*}
Note that if $f_{s_1,t_1}^{D_5}(X)$ splits over $M$ for $(s_1,t_1)\in M^2$ 
then the decomposition type ${\rm DT}(f_{s_1,t_1}^{D_5})$ over $M$ has to be $2,2,1$, 
and $G_{s_1,t_1}\cong C_2$. 
We put 
\[
\delta_{s,t}:=s^2-4s^3+4t-14st-30s^2t-91t^2-34st^2+s^2t^2+40t^3+24st^3+4t^4-4t^5 
\]
as in (\ref{surfd}). 
Then we have 
\begin{lemma}\label{lemqD}
We take the $k$-generic polynomial $f_{s,t}^{D_5}(X)$ and suppose {\rm char} $k\neq 2$.\\
{\rm (1)} For $(s_1,t_1)\in M^2$, there exists $d_1\in M$ such that 
$d_1^2=\delta_{s_1,t_1}$ if and only if $G_{s_1,t_1}\leq C_5$. \\
{\rm (2)} In the case of $G_{s_1,t_1}\not\leq C_5$ for $(s_1,s_2)\in M^2$, 
the quadratic subextension of $L_{s_1,s_2}$ over $M$ is given by $M(\sqrt{\delta_{s_1,t_1}})$. 
\end{lemma}
Note that in the case of $\mathrm{Gal}(f_{s_1,t_1}^{D_5}/M)\leq C_5$, we convert 
$f_{s_1,t_1}^{D_5}(X)$ into the Hashimoto-Tsunogai form $f_{A,B}^{C_5}(X)$ or 
$h_{A,B}^{C_5}(X)$ as in Section \ref{seQuin}, (\ref{deffAB}). 

We take a $k$-generic polynomial for $F_{20}$, 
\begin{align*}
g_{p,r}^{F_{20}}(X)&
=X^5+\Bigl(r^2(p^2+4)-2p-\frac{17}{4}\Bigr)X^4+\Bigl((p^2+4)(3r+1)+\frac{13p}{2}+1\Bigr)X^3\\
&\hspace*{13mm}-\Bigr(r(p^2+4)+\frac{11p}{2}-8\Bigl)X^2+(p-6)X+1\in k(p,q)[X]. 
\end{align*}
Note that, if necessary, we may convert a $k$-generic polynomial $g_{p,r}^{F_{20}}(X)$ 
to another form $f_{p,q}^{F_{20}}(X)$ by putting 
\[
q:=-\frac{5p+8r+2p^2r}{2},\quad \Bigl(r=-\frac{5p+2q}{2(p^2+4)}\Bigr).
\]
For $(p_1,r_1)\in M^2$, if the polynomial $g_{p_1,r_1}^{F_{20}}(X)$ is reducible over $M$, 
the decomposition type ${\rm DT}(g_{p_1,r_1}^{F_{20}})$ over $M$ is $4,1$ or $2,2,1$, and 
$G_{p_1,r_1}\leq C_4$. 
It follows from (\ref{defpq}) that 
\begin{align*}
s=\frac{-1}{4}\Bigl(5p+8r+2p^2r+(2pr+5)\sqrt{p^2+4}\Bigr),\quad 
t=\frac{1}{2}\Bigl(p+\sqrt{p^2+4}\Bigr).
\end{align*}
Hence the quadratic subextension of $\mathrm{Spl}_{k(p,r)} g_{p,r}^{F_{20}}(X)$ over $k(p,r)$ 
is $k(p,r)(\sqrt{p^2+4})$. 
\begin{lemma}\label{lemqE}
We take the $k$-generic polynomial $g_{p,r}^{F_{20}}(X)$ and suppose {\rm char} $k\neq 2$.\\
{\rm (1)} For $(p_1,r_1)\in M^2$, there exists $b\in M$ such that $b^2=p_1^2+4$ if and only if 
$G_{p_1,r_1}\leq D_5$. 
Moreover, in this case of $G_{p_1,r_1}\leq D_5$, the splitting fields of 
$g_{p_1,r_1}^{F_{20}}(X)$ and of $f_{s_1,t_1}^{D_5}(X)$ over $M$ coincide where 
$s_1=-(5p_1+8r_1+2p_1^2r_1+(2p_1r_1+5)b)/4$, $t_1=(p_1+b)/2$\,{\rm ;}\\
{\rm (2)} In the case of $G_{p_1,r_1}\not\leq D_5$, that is, $G_{p_1,r_1}=F_{20}$ or $C_4$, 
the quadratic subextension of $L_{p_1,r_1}$ over $M$ is $M(\sqrt{p_1^2+4})$ for 
$(p_1,r_1)\in M^2$. 
\end{lemma}
By Lemmas \ref{lemqD} and \ref{lemqE}, we may obtain the subquadratic fields of 
the splitting fields of polynomials $f_{s_1,t_1}^{D_5}(X)$ and $g_{p_1,r_1}^{F_{20}}(X)$ 
over $M$ for $s_1,t_1,p_1,r_1\in M$. 
For $g_{p_1,r_1}^{F_{20}}(X), (p_1,r_1)\in M^2$, 
we also get the quartic subfield of $L_{p_1,r_1}$ when $G_{p_1,r_1}\not\leq D_5$. 
By (\ref{surfd}) and (\ref{defpq}), we obtain the quartic equation in $d$: 
\begin{align*}
16d^4-4(p^2+1)(p^2+4)Wd^2+(p^2+4)W^2=0
\end{align*}
where
\begin{align}
W=W_{p,r}:=-199-16p-4(19p+41)r+4(p^2+4)r^2+16(p^2+4)r^3.\label{defW}
\end{align}
The quartic polynomial 
\begin{align}
g_{p,r}^{C_4}(X):=X^4-(p^2+1)(p^2+4)WX^2+(p^2+4)W^2\in k(p,q)[X]\label{gprC4}
\end{align}
gives the $C_4$-extension $k(p,q,d)=k(x,y)^{C_5}=k(A,B)$ of $k(p,q)$; by Kemper's theorem 
\cite{Kem01}, this quartic polynomial is $k$-generic for $C_4$. 
We also see 
\begin{align*}
d^2=\delta'_{p,r}:=W_{p,r}\Bigl((p^4+5p^2+4)+p(p^2+3)\sqrt{p^2+4}\Bigr)\Big/8. 
\end{align*}
\begin{lemma}\label{lemqF}
We take the $k$-generic polynomial $g_{p,r}^{F_{20}}(X)$ and suppose {\rm char} $k\neq 2$.\\
{\rm (1)} If $G_{p_1,r_1}=F_{20}$ for $(p_1,r_1)\in M^2$, then the cyclic quartic subfield of 
$\mathrm{Spl}_M g_{p,r}^{F_{20}}(X)$ over $M$ is given by $M(d_1)$ where $d_1$ is 
a square root of $\delta'_{p_1,r_1}$\,{\rm ;}\\ 
{\rm (2)} For $(p_1,r_1)\in M^2$, we assume that there exists $b\in M$ such that $b^2=p_1^2+4$, 
that is, $G_{p_1,r_1}\leq D_5$. 
Then there exists $d_1\in M$ such that $d_1^2=W_{p_1,r_1}
\big((p_1^4+5p_1^2+4)+p_1(p_1^2+3)b\big)/8$ if and only if $G_{p_1,r_1}\leq C_5$. 
\end{lemma}
In the case where $\mathrm{Gal}(g_{p_1,r_1}^{F_{20}}/M), 
\mathrm{Gal}(g_{p_1',r_1'}^{F_{20}}/M)\not\leq D_5$, that is, $G_\ba,G_\ba'=F_{20}$ or $C_4$, 
we should know the coincidence of the cyclic quartic subfields of $L_{p_1,r_1}$ and of 
$L_{p_1',r_1'}$ over $M$. 
In \cite{HM-2}, we study the field intersection problem of quartic generic polynomials; 
an answer to the field isomorphism problem of $k$-generic polynomial for $C_4$ is given 
as follows. 
First, the following lemma is well-known (cf. \cite[Chapter 2]{JLY02}): 
\begin{lemma}\label{lemD4C4} Assume {\rm char} $k\neq 2$.\\
{\rm (1)} The polynomial $f_{s,t}^{D_4}(X)=X^4+sX^2+t\in k(s,t)[X]$ is $k$-generic for 
$D_4$\,{\rm ;}\\
{\rm (2)} For $(a,b)\in M^2$, $\mathrm{Gal}(f_{a,b}^{D_4}/M)\leq C_4$ if and only if 
there exists $c\in M$ such that $c^2=(a^2-4b)/b$.
\end{lemma}
From Lemma \ref{lemD4C4}, we see that the polynomial 
\[
f_{s,u}^{C_4}(X):=X^4+sX^2+\frac{s^2}{u^2+4}\in k(s,u)[X]
\]
is $k$-generic for $C_4$. 
The discriminant of $f_{s,u}^{C_4}(X)$ with respect to $X$ equals $16\, s^6u^4/(u^2+4)^3$. 
By using the polynomial $f_{s,u}^{C_4}(X)$ above, we get the following theorem: 
\begin{theorem}[\cite{HM-2}]\label{thC4}
We suppose {\rm char} $k\neq 2$ and take the $k$-generic polynomial 
$f_{s,u}^{C_4}(X):=X^4+sX^2+s^2/(u^2+4)\in k(s,u)[X]$ for $C_4$. 
For $\ba=(a,c)$, $\ba'=(a',c')\in M^2$ with $aa'cc'(c^2+4)(c'^2+4)\neq 0$, 
we assume that $c\neq\pm c'$ and $c\neq \pm 4/c'$. 
Then the splitting fields of $f_{a,c}^{C_4}(X)$ and of $f_{a',c'}^{C_4}(X)$ over $M$ 
coincide if and only if the polynomial $F_{\ba,\ba'}(X)=F_{\ba,\ba'}^{+}(X) F_{\ba,\ba'}^{-}(X)$ 
has a linear factor over $M$ where 
\[
F_{\ba,\ba'}^{\pm}(X)=X^4-aa'X^2+\frac{a^2a'^2(c\pm c')^2}{(c^2+4)(c'^2+4)}. 
\]
\end{theorem}
\begin{remark}
We note that the polynomial $F_{\ba,\ba'}(X)$ in Theorem \ref{thC4} is obtained as the resolvent 
polynomial $\mathcal{RP}_{\Theta,D_4,f}(X)$ for a certain $D_4\times D_4'$-primitive 
$D_4''$-invariant with $f(X)=f_\ba^{D_4}(X)f_{\ba'}^{D_4}(X)$. 
The discriminant of $F_{\ba,\ba'}^\pm(X)$ is given by 
$16a^6a'^6(c\pm c')^2(cc'\mp 4)^4/((c^2+4)^3(c'^2+4)^3)$. 
Hence the condition $aa'\neq 0$, $c\neq \pm c'$ and $c\neq \pm 4/c'$ implies that 
$F_{\ba,\ba'}(X)$ has no repeated factors. 
We may assume that the condition $aa'\neq 0$, $c\neq \pm c'$ and $c\neq \pm 4/c'$ 
without loss of generality as in Remark \ref{remGir} (see also \cite{HM-2}). 
\end{remark}
Applying Theorem \ref{thC4} to the polynomial $g_{p,r}^{C_4}(X)$ as in (\ref{gprC4}), 
in the case of 
\[
a=-(p^2+1)(p^2+4)W,\quad b=(p^2+4)W^2,\quad c=p(p^2+3),
\]
we obtain a criterion in terms of the condition within the field $M$ to determine 
whether the subquartic fields of the splitting fields of $g_{p_1,r_1}^{F_{20}}(X)$ 
and of $g_{p_1',r_1'}^{F_{20}}(X)$ coincide or not. 
\begin{remark}
In the case where the field $M$ includes a primitive $4$th root $i:=e^{2\pi \sqrt{-1}/4}$ 
of unity, by Kummer theory, the polynomial $h_t^{C_4}(X):=X^4-t\in k(t)[X]$ is $k$-generic 
for $C_4$. 
Indeed we see that the polynomials $f_{a,c}^{C_4}(X)=X^4+aX^2+a^2/(c^2+4)$ and 
$X^4-a^2(c-2i)/(c+2i)$ are Tschirnhausen equivalent over $M$ because 
\[
\mathrm{Resultant}_X \Bigl(f_{a,c}^{C_4}(X),
Y-\Bigl(\frac{(c+i)(c-2i)}{c}X+\frac{c^2+4}{ac}X^3\Bigr)\Bigr)=Y^4-\frac{a^2(c-2i)}{c+2i}.
\]
In this case, for $b,b'\in M$ with $b\cdot b'\neq 0$, 
the splitting fields of $h_b^{C_4}(X)$ and of $h_{b'}^{C_4}(X)$ over $M$ coincide if and only 
if the polynomial $(X^4-bb')(X^4-b^3b')$ has a linear factor over $M$. 
\end{remark}
\subsection{Reducible and tractable cases of char $k=2$}\label{subR2}
Throughout this subsection we assume that char $k=2$. 
As in the previous subsection, we treat the reducible and the tractable cases 
for a field $k$ of char $k=2$. 

We first note that, by Artin-Schreier theory, 
the polynomial $f_t^{C_2}(X):=X^2+X+t\in k(t)$ is $k$-generic for $C_2$. 
We take the $k$-generic polynomial 
\begin{align*}
f_{s,t}^{D_5}(X)=X^5+(t+1)X^4+(s+t+1)X^3+(t^2+t+1)X^2+sX+t 
\end{align*}
of the Brumer's form for $D_5$ where $k(x,y)^{D_5}=k(s,t)$. 
We denote the constant term of the equation (\ref{eqe}) by 
\[
\epsilon_{s,t}:=\frac{1+s+s^3+t^2+t^4+t^5}{(s+t+st)^2}. 
\]
Then we have 
\begin{lemma}\label{lemqD2}
Assume char $k=2$. 
For the $k$-generic polynomial $f_{s,t}^{D_5}(X)$, the following two assertions hold\,{\rm :}\\
{\rm (1)} For $(s_1,t_1)\in M^2$, there exists $e_1\in M$ such that 
$e_1^2=e_1+\epsilon_{s_1,t_1}$ if and only if $G_{s_1,t_1}\leq C_5;$ \\
{\rm (2)} In the case of $G_{s_1,t_1}\not\leq C_5$ for $(s_1,s_2)\in M^2$, 
the quadratic subextension of $L_{s_1,s_2}$ over $M$ is given as the splitting field of 
$X^2+X+\epsilon_{s_1,t_1}$ over $M$. 
\end{lemma}
For the Frobenius group $F_{20}$ of order $20$, we take the $k$-generic polynomial 
\begin{align*}
f_{p,q}^{F_{20}}(X)&=X^5+\Bigl(\frac{q^2+pq+1}{p^2}\Bigr)X^4+(p^2+p+q+1)X^3+(p+q)X^2+pX+1 
\end{align*}
as in the previous section, where $k(x,y)^{F_{20}}=k(p,q)$. 
From (\ref{defpq}), we get the relations, 
\begin{align}
t^2+pt+1=0,\quad ps+qt+1=0.\label{trant}
\end{align}
It follows from 
\begin{align}
s=(qt+1)/p\label{sols}
\end{align}
that $k(x,y)^{D_5}=k(p,q)(t)$. 
We put 
\begin{align}
T:=t/p\label{defT}
\end{align}
and have $T^2+T+1/p^2=0$. 
Hence the quadratic subfield $k(p,q)(T)$ of $\mathrm{Spl}_{k(p,q)} 
f_{p,q}^{F_{20}}(X)$ over $k(p,q)$ is given as the splitting field of the polynomial 
\[
X^2+X+1/p^2
\]
of the Artin-Schreier type over $k(p,q)$. 
\begin{lemma}\label{lemqE2}
Assume {\rm char} $k=2$.  For the $k$-generic polynomial $f_{p,q}^{F_{20}}(X)$, the following 
two assertions hold\,{\rm :}\\
{\rm (1)} For $(p_1,q_1)\in M^2$, there exists $b\in M$ such that $b^2=b+1/p_1^2$ 
if and only if $G_{p_1,q_1}\leq D_5$\,{\rm ;}\\
{\rm (2)} In the case of $G_{p_1,q_1}\not\leq D_5$, that is, $G_{p_1,q_1}=F_{20}$ or $C_4$, 
the quadratic subextension of $L_{p_1,q_1}$ over $M$ is given as the splitting field of 
$X^2+X+1/p_1^2$ over $M$ for $(p_1,q_1)\in M^2$. 
\end{lemma}
From Lemmas \ref{lemqD2} and \ref{lemqE2}, we obtain the subquadratic fields of 
the splitting fields of polynomials $f_{s_1,t_1}^{D_5}(X)$ and $f_{p_1,q_1}^{F_{20}}(X)$ 
over $M$ for $s_1,t_1,p_1,q_1\in M$. 
We are able to distinguish such quadratic fields by the following well-known lemma 
(cf. \cite{AS26}): 
\begin{lemma}[Artin-Schreier \cite{AS26}]\label{lemAS}
Take the $k$-generic polynomial $f_s^{C_2}(X)=X^2+X+s$ for $C_2$. 
For $a,a'\in M, (a,a'\not\in \{c^2+c\,|\,c\in M\})$, the splitting fields of 
$f_a^{C_2}(X)$ and of $f_{a'}^{C_2}(X)$ over $M$ coincide if and only if the polynomial 
$f_{a+a'}^{C_2}(X)=X^2+X+(a+a')$ has a linear factor over $M$. 
\end{lemma}

Next, we consider when the quartic subfields of $\mathrm{Spl}_M f_{p_1,p_2}^{F_{20}}(X)$ and of 
$\mathrm{Spl}_M f_{p_1',p_2'}^{F_{20}}(X)$ coincide for $(p_1,q_1), (p_1',q_1')\in M^2$ 
under the condition $G_{p_1,q_1}, G_{p_1',q_1'}\not\leq D_5$. 
By Lemma \ref{lemAS}, we should treat only the case where the splitting fields of 
$X^2+X+1/p_1^2$ and of $X^2+X+1/p_1'^2$ over $M$ coincide. 
In this case, we may also assume that $p_1=p_1'=:P_1$ because two polynomials 
$X^2+X+1/p_1^2$ and  $X^2+X+1/p_1'^2$ are Tschirnhausen equivalent over $M$. 

We may use the following classical lemma (\cite{Alb34}): 
\begin{lemma}[Albert \cite{Alb34}, Theorems 4 and 19]\label{lemAlb}
Let $M(x)$ and $M(y)$ be cyclic quartic fields over $M$ with a common quadratic subfield 
$M(u)$ over $M$ so that we may assume 
\begin{align*}
u^2=u+a,\quad x^2=x+(au+b),\quad y^2=y+(au+b)+c
\end{align*}
for some $a,b,c\in M$ with $a\not\in\{d^2+d\,|\,d\in M\}$. 
Then the fields $M(x)$ and $M(y)$ coincide if and only if the polynomial 
$X^2+X+(a+c)$ has a linear factor over $M$. 
\end{lemma}
By the equalities (\ref{eqe}), (\ref{trant}), (\ref{sols}) and (\ref{defT}), we see that 
$k(x,y)^{F_{20}}=k(p,q)$, $k(x,y)^{D_5}=k(p,q)(T)$ and $k(x,y)^{C_5}=k(p,q)(T,e)=k(p,q)(e)$, 
where 
\begin{align*}
&T^2+T+1/p^2=0,\\
&e^2+e+\frac{1+p^3+p^5+pq+p^3q+q^2}{p(1+q)^2}T+\frac{p^4+p^5+p^6+q+p^4q+pq^2+q^3}{p^2(1+q)^2}=0.
\end{align*}
In order to apply Lemma \ref{lemAlb} to the cyclic quartic extension 
$k(p,q)(e)/k(p,q)$, we modify the primitive element $e$ by putting 
\begin{align}
E:=e+\frac{1+p+q+p^3}{p(1+q)}(T+1)+\frac{1+q+p^2}{p^2(1+q)}.
\end{align}
Then we have $k(x,y)^{C_5}=k(p,q)(E)$, $k(x,y)^{D_5}=k(p,q)(T)$ and the equalities 
\begin{align*}
&T^2+T+\frac{1}{p^2}=0,\\
&E^2+E+\frac{T}{p^2}+\frac{1+p+p^4+p^5+q+q^3}{p^2(1+q)^2}=0. 
\end{align*}
Hence we may apply Lemma \ref{lemAlb} to our situation where 
\begin{align*}
u=T,\ a=\frac{1}{P_1^2},\ b=\frac{1+P_1+P_1^4+P_1^5+q_1+q_1^3}{P_1^2(1+q_1)^2},\ 
c=b+\frac{1+P_1+P_1^4+P_1^5+q_1'+q_1'^3}{P_1^2(1+q_1')^2}
\end{align*}
and $P_1:=p_1=p_1'$. 
In particular, we obtain a criterion whether the subquartic fields of 
the splitting fields of $f_{p_1,r_1}^{F_{20}}(X)$ and of $f_{p_1',r_1'}^{F_{20}}(X)$ 
coincide or not, in terms of the condition within the field $M$. 

By the result of Subsections \ref{subR} and \ref{subR2}, 
we reach a partial solution of the field intersection problem of 
$f_{s,t}^{D_5}(X)$ and $f_{p,q}^{F_{20}}(X)$ (or $g_{p,r}^{F_{20}}(X)$ when char $k\neq 2$) 
in the reducible cases, that is, $G_\ba\leq C_4$, and also in the case where 
two splitting fields $L_\ba$ and $L_{\ba'}$ have either a quadratic or a 
quartic subfield over $M$ as the intersection. 
Namely we determined the situation except for the cases 
\{(I-3),(I-4),(I-5)\},\{(II-2),(II-3)\},\{(III-1),(III-2)\} on Table $1$.

\subsection{Dihedral case}\label{subD}

Let $k$ be an arbitrary field. 
In this subsection we investigate the method to distinguish the difference of 
the cases \{(II-2),(II-3)\} and \{(III-1),(III-2)\} on Table $1$. 
We use the $k$-generic polynomial 
\begin{align*}
f_{s,t}^{D_5}(X)=X^5+(t-3)X^4+(s-t+3)X^3+(t^2-t-2s-1)X^2+sX+t
\end{align*}
for $D_5$. 
Note that $k(\bs):=k(s,t)=k(x,y)^{D_5}$ where $s$ and $t$ are given in terms of $x,y$ by 
\begin{align*}
s&=-(x-2x^2+x^3+y-4xy+5x^2y-3x^3y+x^4y-2y^2+5xy^2\\
&\hspace*{9mm} -5x^2y^2+2x^3y^2+y^3-3xy^3+2x^2y^3-x^3y^3+xy^4)/(x^2y^2),\\
t&=-\frac{(x-1)(y-1)(x+y-1)}{xy}
\end{align*}
as in (\ref{defst}). 
We set $(s',t',d'):=\iota(s,t,d)$ and $\bs'=(s',t')$. 
In the case of $f_{s,t}^{D_5}(X)$, we take
\begin{align*}
P&:=\sum_{i=0}^4 (\sigma\sigma')^i(x x')\\
&=xx'+yy'+\frac{(y-1)(y'-1)}{xx'}+\frac{(x+y-1)(x'+y'-1)}{xx'yy'}+\frac{(x-1)(x'-1)}{yy'}
\end{align*}
as a suitable $D_5\times D_5'$-primitive $D_5''$-invariant. 

To distinguish the cases (II-2) and (II-3), we evaluate the formal 
$D_5\times D_5'$-relative $D_5''$-invariant resolvent polynomial by $P$. 
In the case of char $k\neq 2$, the result is given as follows: 
\begin{align*}
F_{\bs,\bs'}^1(X)&:=\mathcal{RP}_{P,D_5\times D_5'}(X)
=\prod_{\opi\in (D_5\times D_5')/D_5''}(X-\pi(P))\\
&\ =\Bigl(G_{\bs,\bs'}^1(X)\Bigr)^2-\frac{d^2d'^2}{4}\Bigl(G_{\bs,\bs'}^2(X)\Bigr)^2\ 
\in k(s,t,s',t')[X]\\
\end{align*}
where
\begin{align}
G_{\bs,\bs'}^1(X)&=X^5-(t-3)({t'}-3)X^4+c_3X^3+\frac{c_2}{2}X^2+\frac{c_1}{2}X+\frac{c_0}{2},
\label{RPD5}\\
G_{\bs,\bs'}^2(X)&=X^2+(t+{t'}-1)X+s-t+{s'}-{t'}+t{t'}+2\nonumber
\end{align}
and $c_3,c_2,c_1,c_0\in k(s,t,{s'},{t'})$ are given by 
\begin{align*}
c_3&=\bigl[2s-21t+3t^2-2t{s'}+t^2{s'}-t^2{t'}\bigr]+31-3s{s'}+5t{t'},\\
c_2&=\bigl[-20s+112t+8st-32t^2+2t^3+5t{s'}-13st{s'}-12t^2{s'}+4t^3{s'}-15st{t'}\\
&+14t^2{t'}+2t^3{t'}+8t^2{s'}{t'}-2t^3{t'}^2\bigr]-102+27s{s'}-119t{t'}
-st{s'}{t'}+6t^2{t'}^2,\\
c_1&=\bigl[32s+2s^2-128t-26st+60t^2+4st^2-8t^3-6s^2{s'}-7t{s'}+38st{s'}+9t^2{s'}-5st^2{s'}\\
&-12t^3{s'}+2t^4{s'}-20t{s'}^2-8st{s'}^2+6t^2{s'}^2+2t^3{s'}^2+2st{t'}-77t^2{t'}
+3st^2{t'}+8t^3{t'}-29t^2{s'}{t'}\\
&+st^2{s'}{t'}+18t^3{s'}{t'}-2st^2{t'}^2+10t^3{t'}^2\bigr]+80-37s{s'}+145t{t'}
-45st{s'}{t'}+24t^2{t'}^2-8t^3{t'}^3,\\
c_0&=\bigl[-16s-2s^2+56t+24st+2s^2t-38t^2-8st^2+8t^3+5s^2{s'}-2t{s'}-38st{s'}-7s^2t{s'}\\
&+5t^2{s'}+13st^2{s'}+8t^3{s'}+2st^3{s'}-4t^4{s'}-21t{s'}^2-11st{s'}^2-2t^2{s'}^2+2st^2{s'}^2
+4t^3{s'}^2\\
&-104st{t'}-33s^2t{t'}+105t^2{t'}+35st^2{t'}+4t^3{t'}+16st^3{t'}-6t^4{t'}-2t^5{t'}-s^2t{s'}{t'}
+36t^2{s'}{t'}\\
&-14st^2{s'}{t'}-6t^3{s'}{t'}+6t^4{s'}{t'}+8t^2{s'}^2{t'}-37st^2{t'}^2+22t^3{t'}^2-2st^3{t'}^2
+8t^4{t'}^2+8t^3{s'}{t'}^2\\
&-2t^4{t'}^3\bigr]-24+14s{s'}-8s^2{s'}^2-224t{t'}+st{s'}{t'}-101t^2{t'}^2-st^2{s'}{t'}^2
-8t^3{t'}^3
\end{align*}
with simplifying notation $\bigl[a\bigr]:=a+\iota(a)$ for $a\in k(s,t,s',t')$. 
It follows from the definition of $\iota$ that $\iota(s,t,s',t')=(s',t',s,t)$. 
We also note that $d^2=\delta_{s,t}\in k(s,t)$ and $d'^2=\delta_{s',t'}\in k(s',t')$ 
where $\delta$ is given by the formula (\ref{surfd}). 

In the case of char $k=2$, the result is 
\begin{align*}
F_{\bs,\bs'}^1(X)=\Bigl(G_{\bs,\bs'}^3(X)\Bigr)^2+G_{\bs,\bs'}^3(X)\cdot G_{\bs,\bs'}^4(X)
+(e^2+e)(e'^2+e')\Bigl(G_{\bs,\bs'}^4(X)\Bigr)^2
\end{align*}
where 
\begin{align*}
G_{\bs,\bs'}^3(X)&=X^5+(t+1)({t'}+1)X^4+d_3X^3+d_2X^2+d_1X+d_0,\\
G_{\bs,\bs'}^4(X)&=(s+t+st)({s'}+{t'}+{s'}{t'})(X^2+(t+{t'}+1)X+s+t+{s'}+{t'}+t{t'})\nonumber
\end{align*}
and $d_3,d_2,d_1,d_0\in k(s,t,{s'},{t'})$ are given by 
\begin{align*}
d_3&=\bigl[t(1+t+t{s'}+t{t'})\bigr]+1+s{s'}+t{t'},\\
d_2&=\bigl[s+t+st+t^3+t{s'}+st{s'}+st{t'}+t^2{t'}+t^3{t'}+t^3{t'}^2\bigr]+1+s{s'}+st{s'}{t'}
+t^2{t'}^2,\\
d_1&=\bigl[s+s^2+t+st+t^2+st^2+s^2{s'}+t{s'}+st{s'}+t^2{s'}+st^2{s'}+t^4{s'}+t^2{s'}^2+t^3{s'}^2\\
&+t^2{t'}+t^3{s'}{t'}+st^2{t'}^2+t^3{t'}^2\bigr]+t{t'}(1+s{s'}),\\
d_0&=\bigl[t(st+st{s'}+st^2{s'}+t{s'}^2+st{s'}^2+t^3{t'}+t^4{t'}+s^2{s'}{t'}+st{s'}{t'}
+t^2{s'}{t'}+t^3{s'}{t'}\\
&+t^2{t'}^2+st^2{t'}^2+t^3{t'}^3)\bigr]+t{t'}(1+s{s'})(1+t{t'}) 
\end{align*}
with simplifying notation $\bigl[a\bigr]:=a+\iota(a)$ for $a\in k(s,t,s',t')$. 
Note that $e^2+e=\epsilon_{s,t}\in k(s,t)$ and $e'^2+e'=\epsilon_{s',t'}\in k(s',t')$ 
where $\epsilon_{s,t}$ is given  above as the constant term in (\ref{eqe}). 

Note that the polynomial $F_{\bs,\bs'}^1(X)$ splits into two factors over the field 
$k(s,t,s',t')(d,d')$ (resp. $k(s,t,s',t')(e,e')$) when char $k\neq 2$ (resp. char $k=2$) as 
\begin{align*}
F_{\bs,\bs'}^1(X)&=\Bigl(G_{\bs,\bs'}^1(X)+\frac{dd'}{2}G_{\bs,\bs'}^2(X)\Bigr)
\Bigl(G_{\bs,\bs'}^1(X)-\frac{dd'}{2}G_{\bs,\bs'}^2(X)\Bigr),\\
F_{\bs,\bs'}^1(X)&=\Bigl(G_{\bs,\bs'}^3(X)+(e+e')G_{\bs,\bs'}^4(X)\Bigr)
\Bigl(G_{\bs,\bs'}^3(X)+(e+e'+1)G_{\bs,\bs'}^4(X)\Bigr).
\end{align*}

From Theorem \ref{thF20F20} and Table $1$, we have to determine when it occurs that 
$\omega_f(\pi(P))\in M$ for some $\opi\in (F_{20}\times F_{20}')/F_{20}''$ 
where $f=f_{s_1,t_1}^{D_5}\cdot f_{s_1',t_1'}^{D_5}$ with $(s_1,t_1),(s_1',t_1')\in M^2$. 
Thus we take an element $\rho(P)\in k(s,t)$ which is conjugate of $P$ 
under the action of $F_{20}\times F_{20}'$ but not under the action of $D_5\times D_5'$. 
We put 
\begin{align*}
F_{\bs,\bs'}^2(X)&:=\mathcal{RP}_{\rho(P),D_5\times D_5'}(X)
=F_{\rho(s),\rho(t),s',t'}^1(X)=\rho\bigl(F_{\bs,\bs'}^1(X)\bigr)
\end{align*}
where $\rho$ acts on $k(s,t)$ as in (\ref{actrho}). 
Then the polynomial $F_{\bs,\bs'}^1(X)\cdot F_{\bs,\bs'}^2(X)$ becomes 
the formal $F_{20}\times D_5'$-relative $D_5''$-invariant resolvent polynomial by $P$. 
We state the result of the dihedral case. 
\begin{theorem}\label{thD5}
We take the $k$-generic polynomial $f_{s,t}^{D_5}(X)$ over an arbitrary field $k$. 
For $\ba=(a_1,a_2)$, $\ba'=(a_1',a_2')\in M^2$, we assume that $G_\ba\geq G_{\ba'}\geq C_5$. 
An answer to the field intersection problem of $f_{s,t}^{D_5}(X)$ 
is given by {\rm DT}$(F_{\ba,\ba'}^1)$ and {\rm DT}$(F_{\ba,\ba'}^2)$ as Table $2$ shows. 
\end{theorem}
\begin{center}
{\rm Table} $2$\vspace*{3mm}\\
{\small 
\begin{tabular}{|c|c|l|l|c|l|l|l|}\hline
$G_\ba$& $G_{\ba'}$ & & GAP ID & $G_{\ba,{\ba'}}$ & & ${\rm DT}(F_{\ba,\ba'}^1)$ 
& ${\rm DT}(F_{\ba,\ba'}^2)$ \\ \hline 
& & (II-1) & $[100,13]$ & $D_5\times D_5$ & $L_\ba\cap L_{\ba'}=M$ & $10$ & $10$ \\ \cline{3-8} 
& \raisebox{-1.6ex}[0cm][0cm]{$D_5$} & (II-2) & $[50,4]$ & $(C_5\times C_5)\rtimes C_2$ & 
$[L_\ba\cap L_{\ba'} : M]=2$ & $5^2$ & $5^2$\\ \cline{3-8}
$D_5$ & & \raisebox{-1.6ex}[0cm][0cm]{(II-3)} & \raisebox{-1.6ex}[0cm][0cm]{$[10,1]$} 
& \raisebox{-1.6ex}[0cm][0cm]{$D_5$} & \raisebox{-1.6ex}[0cm][0cm]{$L_\ba=L_{\ba'}$} 
& $5,2^2,1$ & $5^2$ \\ \cline{7-8}
& & & & & & $5^2$ & $5,2^2,1$ \\ \cline{2-8}
& $C_5$ & (II-4) & $[50,3]$ & $D_5\times C_5$ & $L_\ba\cap L_{\ba'}=M$ & $10$ & $10$ \\ \hline
& & (III-1) & $[25,2]$ & $C_5\times C_5$ & $L_\ba\neq L_{\ba'}$ & $5^2$ & $5^2$\\ \cline{3-8}
$C_5$ & $C_5$ & \raisebox{-1.6ex}[0cm][0cm]{(III-2)} & \raisebox{-1.6ex}[0cm][0cm]{$[5,1]$} 
& \raisebox{-1.6ex}[0cm][0cm]{$C_5$} 
& \raisebox{-1.6ex}[0cm][0cm]{$L_\ba=L_{\ba'}$} & $5,1^5$ & $5^2$ \\ \cline{7-8}
& & & & & & $5^2$ & $5,1^5$ \\ \hline
\end{tabular}
}\vspace*{5mm}
\end{center}
\begin{remark}
In the reducible case, that is, the case of $G_{\ba'}=C_2$, we obtain an answer 
to the subfield problem of $f_{s,t}^{D_5}(X)$ via {\rm DT}$(F_{\ba,\ba'}^i)$ for $i=1,2$ 
as on Table $3$ (cf. Lemma \ref{lemqD}): 
\end{remark}
\begin{center}
{\rm Table} $3$\vspace*{3mm}\\
{\small 
\begin{tabular}{|c|c|l|c|l|l|l|}\hline
$G_\ba$& $G_{\ba'}$ & GAP ID & $G_{\ba,{\ba'}}$ & & ${\rm DT}(F_{\ba,\ba'}^1)$ 
& ${\rm DT}(F_{\ba,\ba'}^2)$ \\ \hline 
\raisebox{-1.6ex}[0cm][0cm]{$D_5$} & & $[20,4]$ & $D_{10}$ 
& $L_\ba\not\supset L_\bb$ & $10$ & $10$ \\ \cline{3-7} 
& & $[10,1]$ & $D_5$ & $L_\ba\supset L_\bb$ & $5^2$ & $5^2$ \\ \cline{1-1}\cline{3-7}
$C_5$ & $C_2$ & $[10,2]$ & $C_{10}$ & $L_\ba\cap L_\bb=M$ & $10$ & $10$\\ \cline{1-1}\cline{3-7} 
\raisebox{-1.6ex}[0cm][0cm]{$C_2$} & & $[4,2]$& $C_2\times C_2$ 
& $L_\ba\neq L_\bb$ & $4^2,2$ & $4^2,2$\\ \cline{3-7}
& & $[2,1]$ & $C_2$ & $L_\ba=L_\bb$ & $2^4,1^2$ & $2^4,1^2$\\ \hline
\end{tabular}
}\vspace*{5mm}
\end{center}
\begin{example}
We take $M=\mathbb{Q}$ and write $L_{s_1,t_1}:=\mathrm{Spl}_\mathbb{Q} f_{s_1,t_1}^{D_5}(X)$ 
and $G_{s_1,t_1}:=\mathrm{Gal}(f_{s_1,t_1}^{D_5}/\mathbb{Q})$ for $(s_1,t_1)\in \mathbb{Q}^2$. 
As in Example \ref{exsame}, for $s_1\in\mathbb{Q}$, two polynomials 
\begin{align*}
f_{s_1,1}^{D_5}(X)&=X^5-2X^4+(s_1+2)X^3-(2s_1+1)X^2+s_1X+1,\\ 
f_{s_1+5,-1}^{D_5}(X)&=X^5-4X^4+(s_1+9)X^3-(2s_1+9)X^2+(s_1+5)X-1
\end{align*} have the same splitting field over $\mathbb{Q}$, 
i.e. $L_{s_1,1}=L_{s_1+5,-1}$. 
By Theorem \ref{thD5}, we also see $L_{0,1}=L_{-1,1}$ because 
$F_{0,1,-1,1}^1(X)$ and $F_{0,1,-1,1}^2(X)$ split over $\mathbb{Q}$ 
into the following irreducible factors: 
\begin{align*}
F_{0,1,-1,1}^1(X)&=(X^5-4X^4-3X^3+23X^2+7X-25)\\
&\quad\cdot(X^5-4X^4-3X^3-24X^2-40X-25),\\
F_{0,1,-1,1}^2(X)&=X(X^2-3X+14)(X^2-5X+18)\\
&\quad\cdot(X^5-8X^4+47X^3-171X^2+299X-235).
\end{align*}
Hence we get $L_{5,-1}=L_{0,1}=L_{-1,1}=L_{4,-1}$ and $G_{5,-1}=G_{0,1}=G_{-1,1}=G_{4,-1}=D_5$. 
The equalities $L_{5,-1}=L_{0,1}$ and $L_{-1,1}=L_{4,-1}$ can be checked via Theorem \ref{thD5} as 
\begin{align*}
F_{5,-1,0,1}^2(X)&=X(X+1)^2(X-3)^2(X^5-4X^4-2X^3-35X^2-38X-47),\\
F_{-1,1,4,-1}^2(X)&=X(X-1)^2(X-7)^2(X^5-16X^4+78X^3-159X^2+190X-611).
\end{align*}
In these cases, the decomposition types $\mathrm{DT}(F_{5,-1,0,1}^2/\mathbb{Q})$ and 
$\mathrm{DT}(F_{-1,1,4,-1}^2/\mathbb{Q})$ should be $5,2^2,1$ 
(cf. Theorem \ref{thfun} about multiple factors). 

For $s_1,s_1'\in \mathbb{Z}$ with $-10000\leq s_1< s_1'\leq 10000$, we see that 
$L_{s_1,1}=L_{s_1',1}$ if and only if $(s_1,s_1')\in X_1\cup X_2$ where
\begin{align*}
X_1&=\{(-6,0),(-1,41),(-94,-10)\},\\
X_2&=\{(-1,0),(-6,-1),(-18,-7),(1,34),(0,41),(-6,41),(-167,-8)\}.
\end{align*}
It can be checked directly that, for $s_1,s_1'\in \mathbb{Z}$ in the range 
$-10000\leq s_1<s_1'\leq 10000$ and 
for each of $i=1,2$, $(s_1,s_1')\in X_i$ if and only if the decomposition type 
$\mathrm{DT}(F_{s_1,1,s_1',1}^i/\mathbb{Q})$ includes $1$. 
Note that if $(s_1,s_1')\in X_1\cup X_2$ then 
$G_{s_1,1}=G_{s_1',1}=D_5$ except for $(s_1,s_1')=(-18,-7)$. 
We see $G_{-18,1}=G_{-7,1}=C_5$ because 
$\mathrm{DT}(F_{18,1,-7,1}^2/\mathbb{Q})$ is $5,1^5$ as follows: 
\begin{align*}
F_{18,1,-7,1}^2(X)=&\ (X+5)(X-6)^2(X+16)(X-17)\\
&\cdot(X^5-8X^4-289X^3+777X^2+7679X-23671). 
\end{align*}
\end{example}
\begin{example}
We take $M=\mathbb{Q}$. 
In \cite{KRY}, Kida-Renault-Yokoyama showed that there exist infinitely many $b\in \mathbb{Q}$ 
such that the polynomials $f_{0,1}^{D_5}(X)$ and $f_{b,1}^{D_5}(X)$ have the same splitting 
field over $\mathbb{Q}$. 
Their method enables us to construct such $b$'s explicitly via rational points of the 
associated elliptic curve (cf. \cite{KRY}). 
They also pointed out that in the range $-400\leq s_1, t_1\leq 400$ there are $25$ pairs 
$(s_1,t_1)\in \mathbb{Z}^2$ such that the splitting fields of 
$f_{0,1}^{D_5}(X)$ and of $f_{s_1,t_1}^{D_5}(X)$ over $\mathbb{Q}$ coincide. 
We may classify the $25$ pairs by the polynomials $F_{0,1,s_1,t_1}^1(X)$ and 
$F_{0,1,s_1,t_1}^2(X)$. 
For $i=1,2$, in the range above, the decomposition type 
$\mathrm{DT}(F_{0,1,s_1,t_1}^i/\mathbb{Q})$ includes $1$ if and only if 
$(s_1,t_1)\in X_i$ where 
\begin{align*}
X_1=\{&(0,1),(4,-1),(4,5),(-6,1),(-24,19),(34,11),(36,-5),\\
&(46,-1),(-188,23),(264,31),(372,-5),(378,43)\},\\
X_2=\{&(-1,-1),(-1,1),(5,-1),(41,1),(-43,5),(47,13),(59,-5),\\
&(59,19),(101,19),(125,-23),(149,11),(155,25),(-169,55)\}.
\end{align*}
By Theorem \ref{thD5}, we see that if $F^1_{0,1,s_1,t_1}(X)$ (resp. $F^2_{0,1,s_1,t_1}(X)$), 
$(s_1,t_1)\in\mathbb{Z}^2$, 
has a root in $\mathbb{Q}$ then $(s_1,t_1)=(2u,2v+1)$ 
(resp. $(s_1,t_1)=(2u+1,2v+1)$) for some $u,v\in\mathbb{Z}$ because $F^1_{0,1,s_1,t_1}(X)\in 
\mathbb{Z}[X]$ splits into irreducible factors over the field $\mathbb{F}_2$ of two elements as 
\begin{align*}
F^1_{0,1,0,0}(X)&=(X^5+X^3+1)^2,\\
F^1_{0,1,0,1}(X)&=X(X+1)^4(X^5+X^2+1),\\
F^1_{0,1,1,0}(X)&=X^{10}+X^7+X^4+X^3+1,\\
F^1_{0,1,1,1}(X)&=(X^5+X^3+1)(X^5+X^3+X^2+X+1)
\end{align*}
and $F^2_{0,1,s_1,t_1}(X)\in \mathbb{Z}[X]$ also splits into irreducible factors over 
$\mathbb{F}_2$ as 
\begin{align*}
F^2_{0,1,0,0}(X)&=(X^5+X^3+1)^2,\\
F^2_{0,1,0,1}(X)&=(X^5+X^3+1)(X^5+X^3+X^2+X+1),\\
F^2_{0,1,1,0}(X)&=X^{10}+X^7+X^6+X^4+X^2+X+1,\\
F^2_{0,1,1,1}(X)&=X^3(X+1)^2(X^5+X^3+X^2+X+1). 
\end{align*}

We do not know, however, whether there exist infinitely many pairs $(s_1,t_1)\in\mathbb{Z}^2$ 
such that $\mathrm{Spl}_\mathbb{Q} f_{0,1}^{D_5}(X)=\mathrm{Spl}_\mathbb{Q} f_{s_1,t_1}^{D_5}(X)$ 
or not. 
By Theorem \ref{thD5}, we checked such pairs $(s_1,t_1)\in\mathbb{Z}^2$ 
in the range $-20000\leq s_1,t_1\leq 20000$ and added 
just $\{(526,41)$, $(952,113)$, $(2302,95)$, $(6466,311)$, $(7180,143)$, $(7480,-169)\}$ 
to $X_1$, and $\{(785,-25)$, $(3881,29)$, $(-11215,299)$, $(19739,-281)\}$ to $X_2$.
\end{example}
\subsection{Cyclic case}\label{subC}
Assume that char $k\neq 2$. 
We take Hashimoto-Tsunogai's $k$-generic polynomial 
\begin{align*}
h_{A,B}^{C_5}(X)&=X^5-\frac{P}{Q^2}(A^2-2A+15B^2+2)X^3+\frac{P^2}{Q^3}(2BX^2-(A-1)X-2B)
\in k(A,B)[X] 
\end{align*}
for $C_5$ where $P=(A^2-A-1)^2+25(A^2+1)B^2+125B^4$, $Q=1-A+7B^2+AB^2$. 

The polynomials $h_{A,B}^{C_5}(X)$ and $f_{A,B}^{C_5}(X)$ have the same splitting 
field over $k(A,B)$, and $f_{A,B}^{C_5}(X)$ is defined by Brumer's form 
$f_{s,t}^{D_5}(X)$ as in (\ref{deffAB}). Therefore, we already have a solution for the field 
isomorphism problem of $f_{A,B}^{C_5}(X)$ and of $h_{A,B}^{C_5}(X)$ from the result of 
the previous subsection. 
In this case we see that the formal resolvent polynomials $F_{\bs,\bs'}^1(X)$ and 
$F_{\bs,\bs'}^2(X)$ split over $k(s,t,s',t')(d,d')$ as 
\[
F_{\bs,\bs'}^1(X)=H_{\bs,\bs'}^1(X)\cdot H_{\bs,\bs'}^3(X),\quad 
F_{\bs,\bs'}^2(X)=H_{\bs,\bs'}^2(X)\cdot H_{\bs,\bs'}^4(X)
\]
where $H_{\bs,\bs'}^i(X)$, $(1\leq i\leq 4)$ is the formal $C_5\times C_5'$-relative 
$C_5''$-resolvent polynomial of degree $5$ by $\rho^{i-1}(P)$ and given by 
\begin{align*}
H_{\bs,\bs'}^{i}(X)&:=\mathcal{RP}_{\rho^{i-1}(P),C_5\times C_5'}(X)
=\rho^{i-1}\bigl(\mathcal{RP}_{P,C_5\times C_5'}(X)\bigr)\\
&=\rho^{i-1}\Bigl(G_{\bs,\bs'}^1(X)-\frac{dd'}{2}G_{\bs,\bs'}^2(X)\Bigr)
\end{align*}
where the polynomials $G_{\bs,\bs'}^1(X)$ and $G_{\bs,\bs'}^2(X)$ are given by (\ref{RPD5}).
\begin{example}
Take the $\mathbb{Q}$-generic polynomial $f_{A,B}^{C_5}(X)$ for $C_5$ and $M=\mathbb{Q}$. 
By Proposition \ref{propisom}, for $\ba=(a,b)$, $\ba'=(a',b')\in \mathbb{Z}^2$, 
if $\ba'=(a,\pm b)$ or $\{\ba,\ba'\}=\{(-1,\pm b),(1,\pm b)\}$ then 
$\mathrm{Spl}_\mathbb{Q} f_{a,b}^{C_5}(X)=\mathrm{Spl}_\mathbb{Q} f_{a',b'}^{C_5}(X)$. 
We also see that $f_{a,0}^{C_5}(X)=(X+a)^2(X+a^2-1)(X+1/(a-1))^2$. 

For $\ba=(a,b)$, $\ba'=(a',b')\in \mathbb{Z}^2$ in the range $-50\leq a, a'\leq 50$, 
$1\leq b\leq b'\leq 50$ with $\ba\neq \ba'$, $\{\ba,\ba'\}\neq\{(-1,b),(1,b)\}$, 
we see that the splitting fields of 
$f_{a,b}^{C_5}(X)$ and of $f_{a',b'}^{C_5}(X)$ over $\mathbb{Q}$ coincide if 
and only if $(a,b,a',b')\in \bigcup_{i=1}^4 X_i$ where
\begin{align*}
X_1&=\{(3,3,23,3),(23,3,3,3),(2,2,-28,14)\},\\\
X_2&=\{(16,2,-12,5),(-33,3,-3,3),(-16,13,34,19)\},\\
X_3&=\{(-3,1,-3,11),(7,3,27,9),(8,11,33,14),(23,5,35,7),(41,11,-15,17)\},\\
X_4&=\{(-2,1,3,2),(4,1,-6,2),(3,1,13,7),(-2,2,18,4),(31,1,-19,7),\\
&\qquad (-3,3,-33,3),(-2,3,43,6),(12,4,46,10)\}.
\end{align*}
By Theorem \ref{thD5}, it can be checked, in the range above and for each of $i=1,2,3,4$, that 
$(a,b,a',b')\in X_i$ if and only if the decomposition type of $H_{\ba,\ba'}^i(X)$ over 
$\mathbb{Q}$ includes $1$. 
\end{example}
\begin{example}
We take $M=\mathbb{Q}(n)$ and regard $n$ as an independent parameter over $\mathbb{Q}$. 
We specialize Hashimoto-Tsunogai's generic polynomials $f_{A,B}^{C_5}(X)$ 
and $h_{A,B}^{C_5}(X)$ by $A:=2n+3$, $B:=1$. 
Then we obtain the cyclic quintic polynomial $f_{2n+3,1}^{C_5}(X)=f_{s,t}^{D_5}(X)$ over 
$\mathbb{Q}(n)$ where $s=n^5+5n^4+12n^3+10n^2-5n-20$, $t=-n^3-5n^2-10n-7$ and 
\[
h_{2n+3,1}^{C_5}(X)=X^5-R(n^2+2n+5)X^3-R^2(X^2+(n+1)X-1)
\]
where $R:=n^4+5n^3+15n^2+25n+25$. 
The discriminants of the polynomials 
$f_{2n+3,1}^{C_5}(X)$ and $h_{2n+3,1}^{C_5}(X)$ with respect to $X$ 
are given by $R^8(n^3+5n^2+10n+7)^2$ and 
$R^8(2n^4+7n^3+23n^2+30n+35)^2(n^6+6n^5+19n^4+34n^3+36n^2+10n-5)^2$ respectively. 

On the other hand, we take Lehmer's simplest quintic polynomial $g_n(X)$ which is given as 
\begin{align*}
g_n(X)=X^5&+n^2X^4-(2n^3+6n^2+10n+10)X^3\\
&+(n^4+5n^3+11n^2+15n+5)X^2+(n^3+4n^2+10n+10)X+1
\end{align*}
(cf. \cite{Leh88}). 
The discriminant of $g_n(X)$ with respect to $X$ is $R^4(n^3+5n^2+10n+7)^2$. 
By using the result in \cite{HR}, we see that if $s=n^5+5n^4+12n^3+10n^2-5n-20$ and 
$t=-n^3-5n^2-10n-7$ then the polynomial $g_n(X)$ and Brumer's quintic $f_{s,t}^{D_5}(X)$ 
has the same splitting field over $\mathbb{Q}(n)$. 
Hence we conclude that the splitting fields of $f_{2n+3,1}^{C_5}(X)$, of $h_{2n+3,1}^{C_5}(X)$ 
and of $g_n(X)$ over $\mathbb{Q}(n)$ coincide. 
By Theorem \ref{thD5}, we checked the pairs $(m,m')\in \mathbb{Z}^2$ in the range 
$-10000\leq m<m'\leq 10000$ to confirm that $\mathrm{Spl}_\mathbb{Q} g_m(X)
=\mathrm{Spl}_\mathbb{Q} g_{m'}(X)$ if and only if $(m,m')=(-2,-1)$. 
\end{example}

\subsection{The case of the Frobenius group $F_{20}$}\label{subF}
Let $k$ be an arbitrary field. 
By the results of the previous subsections, 
we should treat only the remaining three cases \{(I-3),(I-4),(I-5)\}. 

For the Frobenius group $F_{20}$ of order $20$, we take the $k$-generic polynomial 
\begin{align*}
f_{p,q}^{F_{20}}(X)&=X^5+\Bigl(\frac{q^2+5pq-25}{p^2+4}-2p+2\Bigr)X^4\\
&\hspace*{11mm} +\bigl(p^2-p-3q+5\bigr)X^3+(q-3p+8)X^2+(p-6)X+1\ \in\ k(p,q)[X]. 
\end{align*}
In the case of char $k\neq 2$, as we mentioned in Section \ref{seQuin}, 
we may also take a $k$-generic polynomial of the Lecacheux's form for $F_{20}$: 
\begin{align*}
g_{p,r}^{F_{20}}(X)&
=X^5+\Bigl(r^2(p^2+4)-2p-\frac{17}{4}\Bigr)X^4+\Bigl((p^2+4)(3r+1)+\frac{13p}{2}+1\Bigr)X^3\\
&\hspace*{13mm}-\Bigr(r(p^2+4)+\frac{11p}{2}-8\Bigl)X^2+(p-6)X+1\in k(p,q)[X].
\end{align*}
{\bf Method 1.}
Instead of the computation of $\mathcal{R}_{\bs,\bs'}^1(X)$, we construct 
$F_{20}\times F_{20}'$-relative $D_5''$-resolvent polynomial $\mathcal{H}_{\bs,\bs'}(X)$ 
by using the resolvent polynomial $F_{\bs,\bs'}^1(X)$ which is explicitly given in 
Subsection \ref{subD}. 
We put 
\begin{align*}
F_{\bs,\bs'}^3(X)&:=\mathcal{RP}_{\rho'(P),D_5\times D_5'}(X)
=F_{s,t,\rho'(s'),\rho'(t')}^1(X)=\rho'\bigl(F_{\bs,\bs'}^1(X)\bigr),\\
F_{\bs,\bs'}^4(X)&:=\mathcal{RP}_{\rho\rho'(P),D_5\times D_5'}(X)
=F_{\rho(s),\rho(t),\rho'(s'),\rho'(t')}^1(X)=\rho\rho'\bigl(F_{\bs,\bs'}^1(X)\bigr).
\end{align*}
Then the polynomial 
\[\mathcal{H}_{\bs,\bs'}(X):=\prod_{i=1}^4\, F_{\bs,\bs'}^i(X)
\] 
becomes the formal $F_{20}\times F_{20}'$-relative $D_5''$-invariant resolvent polynomial 
$\mathcal{RP}_{P,F_{20}\times F_{20}'}(X)$ by $P$. 
Hence we get the following theorem: 
\begin{theorem}\label{thF20}
We take the $k$-generic polynomial $f_{p,q}^{F_{20}}(X)$ $($or $g_{p,r}^{F_{20}}(X)$ when 
char $k\neq 2)$. 
For $\ba=(a_1,a_2)$, $\ba'=(a_1',a_2')\in M^2$, we assume that $G_\ba\cong G_{\ba'}\cong F_{20}$. 
An answer to the field intersection problem of $f_{p,q}^{F_{20}}(X)$ 
$($or $g_{p,r}^{F_{20}}(X))$ is given by {\rm DT}$(\mathcal{H}_{\ba,\ba'})$ as Table $4$ shows. 
\end{theorem}
\begin{center}
{\rm Table} $4$\vspace*{3mm}\\
{\small 
\begin{tabular}{|c|c|c|l|c|l|l|l|}\hline
$G_\ba$& $G_{\ba'}$ & & GAP ID & $G_{\ba,{\ba'}}$ & & {\rm DT}($\mathcal{R}_{\ba,\ba'}^1)$ 
& {\rm DT}($\mathcal{H}_{\ba,\ba'})$\\ \hline 
& & (I-1) & $[400,205]$ & $F_{20}\times F_{20}$ & $L_\ba\cap L_{\ba'}=M$ 
& $20$ & $40$\\ \cline{3-8} 
& & (I-2) & $[200,42]$ &$(D_5\times D_5)\rtimes C_2$ & $[L_\ba\cap L_{\ba'}:M]=2$ 
& $10^2$ & $20^2$\\ \cline{3-8} 
$F_{20}$ & $F_{20}$ & (1-3) & $[100,11]$ & $(C_5\times C_5)\rtimes C_4$ 
& $[L_\ba\cap L_{\ba'}:M]=4$ & $5^4$ & $10^4$\\ \cline{3-8}
& & (I-4) & $[100,12]$ 
& $(C_5\times C_5)\rtimes C_4$ & $[L_\ba\cap L_{\ba'}:M]=4$ 
& $10^2$ & $10^4$\\ \cline{3-8} 
& & (I-5) & $[20,3]$ & $F_{20}$ & $L_\ba=L_{\ba'}$ & $5^3,4,1$ & $10^3,4^2,2$\\ \cline{1-8}
\end{tabular}
}\vspace*{5mm}
\end{center}
{\bf Method 2.}
We may apply the result of the dihedral case in Subsection \ref{subD}. 
Indeed, in the case of char $k\neq 2$, we may convert the Brumer's 
polynomial $f_{s,t}^{D_5}$ to 
$f_{p,q}^{F_{20}}(X)$ and $g_{p,r}^{F_{20}}(X)$, respectively, by 
\begin{align*}
s=\frac{-10+q\bigl(p+\sqrt{p^2+4}\bigr)}{2\sqrt{p^2+4}},\qquad 
t=\frac{1}{2}\Bigl(p+\sqrt{p^2+4}\Bigr)
\end{align*}
and
\begin{align}
s=\frac{-1}{4}\Bigl(5p+8r+2p^2r+(2pr+5)\sqrt{p^2+4}\Bigr),\quad 
t=\frac{1}{2}\Bigl(p+\sqrt{p^2+4}\Bigr). \label{prtost}
\end{align}
In general, we need the factoring process of a resolvent polynomial over the biquadratic 
extension $M(\sqrt{p^2+4},\sqrt{p'^2+4})$ of $M$. 
However, because we should only treat the case $M(\sqrt{p^2+4})=M(\sqrt{p'^2+4})$, 
all we need is the factoring algorithm over $M(\sqrt{p^2+4})$. 
This is also feasible in the case of char $k=2$ by using the result of Subsection \ref{subR2}. 
\begin{example}
We take $M=\mathbb{Q}$ and the $\mathbb{Q}$-generic polynomial $g_{p,r}^{F_{20}}(X)$ 
of the Lecacheux's form for $F_{20}$. 
We first see that 
\[
g_{p,2}^{F_{20}}(X)=\Bigl(X-\frac{1}{4}\Bigr)
\Bigl(X^4+2(2p^2-p+6)X^3+2(4p^2+3p+16)X^2-4(p-2)X-4\Bigr)
\]
and the splitting fields of $g_{p,2}^{F_{20}}(X)$ and of $g_{-p,2}^{F_{20}}(X)$ over $k(p)$ 
coincide. 
From (\ref{prtost}), we also see that $\mathrm{Gal}(g_{0,r}^{F_{20}}(X)/k(r))\leq D_5$ 
because the splitting fields of $g_{0,r}^{F_{20}}(X)$ and of $f_{-(4r+5)/2,1}^{D_5}(X)$ over 
$k(r)$ coincide. 

For $\bp=(p_1,r_1)$, $\bp'=(p_1',r_1')\in \mathbb{Z}^2$ in the range $-100\leq p_1,p_1'\leq 100$, 
$-100\leq r_1\leq r_1'\leq 100$ with $\bp\neq \bp'$, $r_1,r_1' \neq 2$, 
we see that the splitting fields of 
$g_\bp^{F_{20}}(X)$ and of $g_{\bp'}^{F_{20}}(X)$ over $\mathbb{Q}$ coincide if 
and only if $(p_1,r_1,p_1',r_1')\in X_1\cup X_2$ where
\begin{align*}
X_1&=\{(-3,-3,3,0),(1,-8,-1,-1),(11,1,11,7),(-1,10,11,22),(-1,-11,29,0)\},\\
X_2&=\{(7,1,-7,4),(11,1,11,13),(11,7,11,13),(11,12,11,62),(11,31,11,73),(-2,6,-2,84)\}.
\end{align*}
Using Method $2$, it can be checked by Theorem \ref{thD5}, 
in the range above and for each of $i=1,2$, that $(p_1,r_1,p_1',r_1')\in X_i$ 
if and only if two quadratic fields $\mathbb{Q}(\sqrt{p_1^2+4})$ and 
$\mathbb{Q}(\sqrt{p_1'^2+4})$ coincide 
and the decomposition type of $F_{\bp,\bp'}^i(X)$ over $\mathbb{Q}(\sqrt{p_1^2+4})$ 
includes $1$. 
We note that $\mathrm{Gal}(g_\bp^{F_{20}}(X)/\mathbb{Q})\cong 
\mathrm{Gal}(g_{\bp'}^{F_{20}}(X)/\mathbb{Q})\cong F_{20}$ for each $(p_1,r_1,p_1',r_1')\in 
X_1\cup X_2$. 
\end{example}
%

\vspace*{1mm}
\begin{acknowledgment}
We thank Professor Kazuhiro Yokoyama for drawing our attention to multi-resolvent 
polynomials and also for his useful suggestions. 
\end{acknowledgment}
\vspace*{1mm}

{\small 
\hspace*{-0.5cm}\\ 
\begin{tabular}{ll}
Akinari HOSHI & Katsuya MIYAKE\\ 
Department of Mathematics & Department of Mathematics\\ 
Faculty of Science & School of Fundamental Science and Engineering\\ 
Rikkyo University & Waseda University\\
3--34--1 Nishi Ikebukuro Toshima-ku & 3--4--1 Ohkubo Shinjuku-ku\\
Tokyo, 171--8501, Japan & Tokyo, 169--8555, Japan\\
E-mail: \texttt{hoshi@rikkyo.ac.jp} & E-mail: \texttt{miyakek@aoni.waseda.jp}
\end{tabular}
}

\end{document}